\documentclass[12pt, a4paper]{article}

\usepackage{amsmath, amssymb, amsthm, amsrefs}
\usepackage{caption}
\usepackage{enumitem}
\usepackage{graphicx}
\usepackage{geometry, color}

\numberwithin{equation}{section}

\theoremstyle{plain}
\newtheorem{thm}{Theorem}[section]

\newtheorem{lem}[thm]{Lemma}
\newtheorem{prop}[thm]{Proposition}

\title{A perturbation problem involving singular perturbations of domains for Hamilton-Jacobi equations}

\author{Taiga Kumagai}

\date{}

\begin{document}

\maketitle

\begin{abstract}
We investigate a singular perturbation for Hamilton-Jacobi equations in an open subset of two dimensional Euclidean space,
where the set is determined through a Hamiltonian and the Hamilton-Jacobi equations are
the dynamic programming equations for optimal control of the Hamiltonian flow of the Hamiltonian.
We establish the convergence of solutions of the Hamilton-Jacobi equations  
and identify the limit of the solutions as the solution of systems of ordinary differential equations on a graph.
The perturbation is singular in the sense that the domain degenerates to the graph in the limiting process.
Our result can be seen as a perturbation analysis, in the viewpoint of optimal control, of the Hamiltonian flow.

\textit{Key Words and Phrases.} Singular perturbation, Hamilton-Jacobi equations, Singular perturbation of domains.

2010 \textit{Mathematics Subject Classification Numbers.} 35B40, 49L25.
\end{abstract}

\section{Introduction}

In this paper we consider the asymptotic behavior, as $\varepsilon \to 0+$,
of the solution $u^\varepsilon$ of the boundary value problem
for the Hamilton-Jacobi equation, with a small parameter $\varepsilon > 0$,
\begin{align} \label{epHJ} 
\begin{cases} 
\lambda u^\varepsilon - \cfrac{b \cdot Du^\varepsilon}{\varepsilon} + G(x, Du^\varepsilon ) = 0 \ \ \ &\text{ in } \Omega, \tag{$\mathrm{HJ}^\varepsilon$} \\ 
u^\varepsilon = g^\varepsilon \ \ \ &\text{ on } \partial \Omega.
\end{cases}
\end{align}

Here $\lambda$ is a positive constant,
$\Omega$ is an open subset of $\mathbb R^2$ with boundary $\partial \Omega$,
$u^\varepsilon : \overline \Omega \to \mathbb R$ is the unknown,
$G : \overline \Omega \times \mathbb R^2 \to \mathbb R$ and $g^\varepsilon : \partial \Omega \to \mathbb R$ are given functions, 
and $b : \mathbb R^2 \to \mathbb R^2$ is a Hamiltonian vector field,
that is, for a given Hamiltonian $H : \mathbb R^2 \to \mathbb R$,
\begin{equation*}
b = (H_{x_2}, -H_{x_1}),
\end{equation*}
where the subscript $x_i$ indicates
the differentiation with respect to the variable $x_i$.

Let us consider an optimal control problem,
where the state is described by the initial value problem
\begin{align} \label{state}
\begin{cases}
\dot X^\varepsilon (t) = \cfrac{1}{\varepsilon} \, b(X^\varepsilon (t)) + \alpha (t) \ \ \ \text{ for } t \in \mathbb R, \\
X^\varepsilon (0) = x \in \mathbb R^2, \\
\alpha \in L^\infty (\mathbb R ; \mathbb R^2), 
\end{cases}
\end{align}
of which the solution will be denoted by $X^\varepsilon (t, x, \alpha)$,
the discount rate and pay-off are given by $\lambda$ and $g^\varepsilon$, respectively, 
and the running cost is given by the function $L$,
called \textit{Lagrangian} of $G$, defined by
\begin{equation*}
L(x, \xi ) = \sup_{p \in \mathbb R^2} \{ -\xi \cdot p - G(x, p) \} \ \ \ \text{ for } (x, \xi ) \in \overline \Omega \times \mathbb R^2.
\end{equation*} 
Then \eqref{epHJ} is the dynamic programming equation for this optimal control problem.

We may regard \eqref{state} as an equation
obtained by perturbing by control $\alpha$ the Hamiltonian system 
\begin{align}
\begin{cases} \label{epHS}
\dot X^\varepsilon (t) = \cfrac{1}{\varepsilon} \, b(X^\varepsilon (t)) \ \ \ \text{ for } t \in \mathbb R, \\ \tag{$\mathrm{HS}^\varepsilon$}
X^\varepsilon (0) = x \in \mathbb R^2,
\end{cases}
\end{align}
of which the solution will be denoted by $X^\varepsilon (t, x)$.

Perturbation problems for Hamiltonian flows,
similar to ours, were studied by Freidlin-Wentzell \cites{FW}. 
In their work \cites{FW}, they considered the stochastic perturbation of \eqref{epHS}
\begin{align} \label{SDE}
\begin{cases}
dX_t^\varepsilon  = \cfrac{1}{\varepsilon} \, b(X_t^\varepsilon ) dt + dW_t \ \ \ \text{ for } t \in \mathbb R, \\
X(0) = x \in \mathbb R^2,
\end{cases}
\end{align}  
where $W_t$ is a standard two dimensional Brownian motion,
and associated with \eqref{SDE} is, in place of \eqref{epHJ},
the boundary value problem for the linear second-order elliptic partial differential equation
\begin{align} \label{LDE}
\begin{cases}
- \, \cfrac{1}{2} \, \Delta u^\varepsilon - \cfrac{b \cdot Du^\varepsilon}{\varepsilon} = f \ \ \ &\text{ in } \Omega, \\
u^\varepsilon = g^\varepsilon \ \ \ &\text{ on } \partial \Omega,
\end{cases}
\end{align}
where $f \in C(\overline \Omega)$ is a given function.

In \cites{FW}, authors proved that solutions of \eqref{LDE} converge,
as $\varepsilon \to 0+$, to solutions of a boundary value problem
for a system of ordinary differential equations (odes, for short)
on a graph by a probabilistic approach.
The domain $\Omega$ thus degenerates to a graph in the limiting process,
which we call \textit{singular perturbations of domains}.
More precisely, the limit of the function $u^\varepsilon$, as $\varepsilon \to 0+$, 
becomes a function which is constant along each component of level sets of the Hamiltonian $H$.
As a consequence, a natural parametrization to describe the limiting function is
to use the height of the Hamiltonian $H$ and not the original two-dimensional variables (see Fig. 2 below).
After work of \cites{FW}, such problems have been studied by many authors, including Freidlin-Wentzell \cites{FW},
and many of them studied by the probabilistic approach. 
In a recent study, Ishii-Souganidis \cites{IS} obtained, by using pure pde-techniques, 
similar results for linear second-order degenerate elliptic partial differential equations.

In the view point of pde, 
linear differential equations, like \eqref{LDE}, give the basis 
for studying stochastic perturbations of \eqref{epHS},
while nonlinear differential equations, like \eqref{epHJ}, play the same role 
for perturbations of \eqref{epHS} by control terms.

Here we treat \eqref{epHJ}, establish the convergence
of the solution $u^\varepsilon$ of \eqref{epHJ} to a function on a graph
and identify the limit of $u^\varepsilon$ as a unique solution of a boundary value problem 
for a system of odes on the graph.
The result is stated in Theorem \ref{main}. 
The argument for establishing this result depends heavily on
viscosity solution techniques including the perturbed test function method 
as well as representations, as value functions in optimal control, of solutions of \eqref{epHJ}.

In \cites{AT}, authors treat a problem similar to the above. 
They consider general Hamilton-Jacobi equations in optimal control 
on an unbounded thin set converging to a graph,
prove the convergence of the solutions, and identify the limit of the solutions.

An interesting point of our result lies in that we have to treat 
a non-coercive Hamiltonian in Hamilton-Jacobi equation \eqref{epHJ}. 
Many authors in their studies on Hamilton-Jacobi equations
on graphs, including \cites{AT}, 
assumes, in order to guarantee the existence of continuous solutions, 
a certain coercivity of the Hamiltonian in the Hamilton-Jacobi equations, which corresponds, in terms of 
optimal control, a certain controllability of the dynamics.  
In our result, we also make a coercivity assumption (see (G4) below) 
on the unperturbed Hamiltonian, called $G$ in \eqref{epHJ},  
but, because of the term $b \cdot Du^\varepsilon/\varepsilon$ in \eqref{epHJ},
when $\varepsilon > 0$ is very small, the perturbed Hamiltonian becomes non-coercive
and the controllability of the dynamics \eqref{state} breaks down. 
A crucial point in our study is that, when $\varepsilon$ is very small, 
the perturbed term $b \cdot Du^\varepsilon/\varepsilon$ 
makes the solution $u^\varepsilon$ nearly constant along the level set of the 
Hamiltonian $H$ while the perturbed Hamiltonian 
$-b(x) \cdot p/\varepsilon + G(x,p)$ in 
\eqref{epHJ} is ``coercive in the direction orthogonal to $b(x)$'', that is 
the direction of the gradient $DH(x)$. 
Heuristically at least, these two characteristics combined together allow us 
to analyze the asymptotic behavior of the solution $u^\varepsilon$ of 
\eqref{epHJ} as $\varepsilon \to 0+$.

This paper is organized as follows.
In the next section, we describe precisely the 
Hamitonian $H$, the domain $\Omega$ as well as 
some relevant properties of the Hamiltonian system \eqref{HS}, 
present the assumptions on the unperturbed Hamiltonian $G$ used throughout the paper,
and give a basic existence and uniqueness proposition (see Proposition \ref{viscosity-sol}) 
for \eqref{epHJ} and a proposition concerning the dynamic programming principle.  
Section 3, which is divided into two parts, is devoted to establishing Theorem \ref{main}.
In the first part, we give some observations on the odes 
\eqref{lim-HJ} below in Section 3 on the graph,
which the limiting function of the solution $u^\varepsilon$ of \eqref{epHJ} should satisfy. 
The last part is devoted to the proof of Theorem \ref{main}. 
It relies on three propositions.
They are the characterization of the half relaxed-limits of $u^\varepsilon$
(Theorem \ref{characterize} in Section 4),
and two estimates for the half relaxed-limits of $u^\varepsilon$
(Lemmas \ref{v^+-u_i^+} and \ref{v^--d_0} in Sections 5 and 6, respectively). 
In Section 7, we are concerned with the admissibility of boundary data for the odes on the graph. 
In our formulation of the main result (see Theorem \ref{main}), 
we assume rather implicit (or ad hoc) assumptions (G5) and (G6). 
Indeed, (G5) readily gives us a unique viscosity solution of \eqref{epHJ} 
and (G6) essentially assumes no boundary-layer phenomenon for the solution of 
\eqref{epHJ} in the limiting process as $\varepsilon \to 0+$. 
It is thus important to know when (G5) and (G6) hold. 
Section 7 focuses to give a sufficient condition under which (G5) and (G6) hold.

Before closing the introduction, we give a few of our notations.

\subsection*{Notation}

For $c, d \in \mathbb R$, we write $c \wedge d = \min \{ c, d \}$ and $c \vee d = \max \{ c, d \}$. 
For $r > 0$, we denote by $B_r$ the open disc centered at the origin with radius $r$.
We write $\mathbf{1}_E$ for the characteristic function of the set $E$.

\section{Preliminaries}

\begin{figure}[t]
\centering
\input{Hamiltonian.tex}
\caption{\ }
\end{figure}

\subsection{The domain $\Omega$}

We assume the following (H1)--(H3) throughout this paper.

\begin{itemize}
\item[(H1)] $H \in C^2 (\mathbb R^2)$ and $\lim_{|x| \to \infty} H(x) = \infty$.

\item[(H2)] $H$ has exactly three critical points $z_1, z_3 \in \mathbb R^2$ 
and $z_2 = (0,0) := 0$.

\item[(H3)] There exists $\kappa > 0$ such that 
                \begin{equation*}
				H(x_1, x_2) = x_2^2 -x_1^2 \ \ \ \text{ on } \overline B_\kappa.
                \end{equation*}
\end{itemize}
The graph of the Hamiltonian $H$ 
satisfying (H1)--(H3) is depicted in Fig. 1 above.

It follows from these assumptions that
for any $h > 0$, the open set $\{ x \in \mathbb R^2 \ | \ H(x) < h \}$ is connected,
and the open set $\{ x \in \mathbb R^2 \ | \ H(x) < 0 \}$ consists of
two connected components $D_1$ and $D_3$ such that $z_1 \in D_1$ and $z_3 \in D_3$.
We may assume that $(-\kappa, 0) \in D_1$ and $(\kappa, 0) \in D_3$.

The shape of the domain $\Omega$ is depicted in Fig. 2(a) below.

We choose $h_1, h_2, h_3 \in \mathbb R$ so that
\begin{equation*}
h_1, h_3 < 0 < h_2 \ \ \ \text{ and } \ \ \ H(z_i) < h_i \ \ \ \text{ for } i \in \{ 1, 3 \}, 
\end{equation*}
and consider the intervals
\begin{equation*}
J_2 = (0, h_2) \ \ \ \text{ and } \ \ \ J_i = (h_i, 0) \ \ \ \text{ for } i \in \{ 1, 3 \}, 
\end{equation*}
the open sets
\begin{equation*}
\Omega_2 = \{ x \in \mathbb R^2 \ | \ H(x) \in J_2 \} \ \ \ \text{ and } \ \ \ \Omega_i = \{ x \in D_i \ | \ H(x) \in J_i \} \ \ \ \text{ for } i \in \{ 1, 3 \}, 
\end{equation*}
and their ``outer" boundaries
\begin{equation*}
\partial_i \Omega = \{ x \in \overline \Omega_i \ | \ H(x) = h_i \} \ \ \ \text{ for } i \in \{ 1, 2, 3 \}.
\end{equation*}

Now we introduce $\Omega$ as the open connected set 
\begin{equation*}
\Omega = \left( \bigcup_{i = 1}^3 \Omega_i \right) \cup \{ x \in \mathbb R^2 \ | \ H(x) = 0 \},
\end{equation*}
with the boundary
\begin{equation*}
\partial \Omega = \bigcup_{i = 1}^3 \partial_i \Omega.
\end{equation*}
It is obvious that, by replacing $\kappa$ by a smaller positive number if necessary,
we may assume that $\overline B_\kappa \subset \Omega$ in (H3).
For later convenience, the constant $\kappa > 0$ in (H3) will be always assumed
small enough so that $\overline B_\kappa \subset \Omega$.

We define loops $c_i (h)$ for $h \in \bar J_i$ and $i \in \{ 1, 2, 3 \}$ by
\begin{equation*}
c_i (h) = \{ x \in \overline \Omega_i \ | \ H(x) = h \}.
\end{equation*}

If we identify all points belonging to a loop $c_i (h)$,
then we obtain a graph, which is shown in Fig. 2(b),
consisting of three segments, parametrized by $J_1$, $J_2$, and $J_3$.

\begin{figure}[t]
\centering
\input{Omega.tex}
\caption{\ }
\end{figure}

We consider
the initial value problem
\begin{align} \label{HS}
\begin{cases}
\dot X (t) = b(X(t)) \ \ \ \text{ for } t \in \mathbb R, \\ \tag{HS}
X(0) = x \in \mathbb R^2,
\end{cases}
\end{align}
which admits a unique global in time solution $X = X(t, x)$.
Note that, in view of (H1),
\begin{equation*}
X, \dot X \in C^1 (\mathbb R \times \mathbb R^2 ; \mathbb R^2) \ \ \ \text{ and } \ \ \ H(X(t, x)) = H(x) \ \ \ \text{ for all } (t, x) \in \mathbb R \times \mathbb R^2. 
\end{equation*}

Fix $i \in \{ 1, 2, 3 \}$ and $h \in \bar J_i \setminus \{ 0 \}$.
Since $\{ X(t, x) \ | \ t \in \mathbb R \} \subset c_i (h)$ if $x \in c_i (h)$,
and $DH(x) \not= 0$ for all $x \in c_i (h)$, 
it is easily seen that the map $t \mapsto X(t, x)$ is periodic in $t$ for all $x \in c_i (h)$,
and that the minimal period of $X(\cdot, x)$ is independent of $x \in c_i (h)$.
For $x \in c_i (h)$, let $T_i (h)$ be the minimal period of $X(\cdot, x)$, that is,
\begin{equation*}
T_i (h) = \inf \{ t > 0 \ | \ X(t, x) = x \}.
\end{equation*}

\begin{lem} \label{T_i}
For all $i \in \{ 1, 2, 3 \}$, $T_i \in C^1 (\bar J_i \setminus \{ 0 \})$ 
and $T_i (h) = O(|\log |h||)$ as $(-1)^ih \to 0+$.
\end{lem}

The proposition above can be found in \cites{IS} in a slightly different context,
to which we refer for the proof.

Finally we observe that, for any $i \in \{ 1, 2, 3 \}$ and $h \in \bar J_i \setminus \{ 0 \}$,
the length $L_i (h)$ of the loop $c_i (h)$ is described by
\begin{equation} \label{def-L_i}
L_i (h) = \int_0^{T_i (h)} |\dot X(t, x)| \, dt = \int_0^{T_i (h)} |DH(X(t, x))| \, dt,
\end{equation}
where $x \in \overline \Omega$ is chosen so that $x \in c_i (h)$.
The value of the integral above is independent of the choice of $x \in c_i (h)$.

The following lemma is a consequence of Lemma \ref{T_i} and \eqref{def-L_i}.

\begin{lem} \label{L_i}
For all $i \in \{ 1, 2, 3 \}$, $L_i \in C^1 (\bar J_i \setminus \{ 0 \})$.
\end{lem}

We focus on the domain given by Hamiltonian $H$
satisfying (H1)--(H3) (see Fig. 1) in this paper for simplicity of presentation. 
It is possible to treat many junctions if Hamiltonian $H$ has only critical points that are non-degenerate. 
However, the argument in this paper does not cover the case when critical points of Hamiltonian $H$ are degenerate,
and thus junctions with more than three line segments are outside the scope of this paper.

\subsection{Viscosity solutions of \eqref{epHJ}}

We prove here the existence and uniqueness of viscosity solutions of \eqref{epHJ},
which are continuous up to the boundary. 
We do not recall here the definition and basic properties of viscosity solutions
and we refer instead to \cites{BCD, CIL, L} for them.

We need the following assumptions on $G$ and $g^\varepsilon$.

\begin{itemize}
\item[(G1)] $G \in C(\overline \Omega \times \mathbb R^2)$.

\item[(G2)] There exists a modulus $m$ such that
                \begin{equation*}
                |G(x, p) - G(y, p)| \leq m(|x - y|(1 + |p|)) \ \ \ \text{ for all } x, y \in \overline \Omega \text{ and } p \in \mathbb R^2.
                \end{equation*}

\item[(G3)] For each $x \in \overline \Omega$, the function $p \mapsto G(x, p)$ is convex on $\mathbb R^2$.

\item[(G4)] $G$ is coercive, that is, 
               \begin{equation*}
               G(x, p) \to \infty \ \ \ \text{ uniformly for } x \in \overline \Omega \text{ as } |p| \to \infty.
               \end{equation*}
\end{itemize}

Under assumptions (G1), (G3), and (G4), there exist $\nu, M > 0$ such that
\begin{equation*}
G(x, p) \geq \nu |p| - M \ \ \ \text{ for  all } (x, p) \in \overline \Omega \times \mathbb R^2,
\end{equation*}
which yields, together with the definition of $L$,
\begin{equation} \label{upper-bound-L}
L(x, \xi ) \leq M \ \ \ \text{ for all } (x, \xi) \in \overline \Omega \times \overline B_\nu.
\end{equation}
Also, by the definition of $L$, we get
\begin{equation} \label{lower-bound-L}
L(x, \xi ) \geq -G(x, 0) \ \ \ \text{ for all } (x, \xi) \in \overline \Omega \times \mathbb R^2, 
\end{equation}
and, using (G1), we see that $L$ is lower semicontinuous
in $\overline \Omega \times \mathbb R^2$.

\begin{itemize}  
\item[(G5)] There exists $\varepsilon_0 \in (0, 1)$ such that
				$\{ g^\varepsilon \}_{\varepsilon \in (0, \varepsilon_0)} \subset C(\partial \Omega)$
                is uniformly bounded on $\partial \Omega$,
                and, for any $\varepsilon \in (0, \varepsilon_0)$, $g^\varepsilon$ satisfies the condition
				\begin{equation*}
                g^\varepsilon (x) \leq \int_0^{\vartheta} L(X^\varepsilon (t, x, \alpha ), \alpha (t)) e^{-\lambda t} \, dt
				+ g^\varepsilon (y) e^{-\lambda \vartheta} 
                \end{equation*}
                for all $x, y \in \partial \Omega$, $\vartheta \in [0, \infty)$, and $\alpha \in L^\infty (\mathbb R; \mathbb R^2)$,
                where the conditions
                \begin{equation*}
                X^\varepsilon (\vartheta, x, \alpha) = y \ \ \ \text{ and } \ \ \
				X^\varepsilon (t, x, \alpha) \in \overline \Omega \ \ \ \text{ for all } t \in [0, \vartheta]
                \end{equation*}
                are satisfied, that is, $\vartheta$ is a visiting time at the target $y$
				of the trajectory $\{ X^\varepsilon (t, x, \alpha) \}_{t \geq 0}$ constrained in $\overline \Omega$.
\end{itemize}

Condition (G5) is a sort of compatibility condition,
and we are motivated to introduce this condition by \cites{L},
where conditions similar to (G5) are used to guarantee
the continuity of the value functions in optimal control.
Here (G5) has the same role as those in \cites{L} and is to ensure the continuity
of the function $u^\varepsilon$ given by \eqref{represent} below.

\begin{prop} \label{viscosity-sol}
Assume that {\rm (G1)--(G5)} hold.
For $\varepsilon \in (0, \varepsilon_0)$,
we define the function $u^\varepsilon : \overline \Omega \to \mathbb R$ by
\begin{align} \label{represent}
\begin{aligned}
u^\varepsilon (x) = \inf \Big\{ \int_0^{\tau^\varepsilon} L(X^\varepsilon (t, x, \alpha), \alpha &(t))e^{-\lambda t} \, dt \\
                         &+ g^\varepsilon (X^\varepsilon (\tau^\varepsilon, x, \alpha))e^{-\lambda \tau^\varepsilon}
						 \ | \ \alpha \in L^\infty (\mathbb R; \mathbb R^2) \Big\},
\end{aligned}
\end{align}
where $\tau^\varepsilon$ is a visiting time in $\partial \Omega$
of the trajectory $\{ X^\varepsilon (t, x, \alpha) \}_{t \geq 0}$ constrained in $\overline \Omega$,
that is, $\tau^\varepsilon$ is a nonnegative number such that
\begin{equation*}
X^\varepsilon (\tau^\varepsilon, x, \alpha) \in \partial \Omega \ \ \ \text{ and } \ \ \
X^\varepsilon (t, x, \alpha) \in \overline \Omega \ \ \ \text{ for all } t \in [0, \tau^\varepsilon].
\end{equation*} 
Then $u^\varepsilon$ is continuous on $\overline \Omega$,
the unique viscosity solution of \textrm{\eqref{epHJ}} and
satisfies $u^\varepsilon = g^\varepsilon$ on $\partial \Omega$.
Furthermore the family $\{ u^\varepsilon \}_{\varepsilon \in (0, \varepsilon_0)}$
is uniformly bounded on $\overline \Omega$.
\end{prop}

\begin{proof}
We begin by showing that
$\{ u^\varepsilon \}_{\varepsilon \in (0, \varepsilon_0)}$ is uniformly bounded on $\overline \Omega$.
Fix a constant $C > 0$ so that
$|g^\varepsilon (x)| \leq C$ for all $(x, \varepsilon) \in \partial \Omega \times (0, \varepsilon_0)$.
We may assume as well that $|G(x, 0)| \leq C$ for all $x \in \overline \Omega$.

We intend to define $\tau^\varepsilon := \tau^\varepsilon (x) \in [0, \infty)$
and $Y^\varepsilon : [0, \tau^\varepsilon] \times \overline \Omega \to \mathbb R^2$.
Let $\varepsilon \in (0, \varepsilon_0)$ and $x \in \overline \Omega$.
If $x \in \partial \Omega$, then we take $\tau^\varepsilon = 0$ and set $Y^\varepsilon (0, x) = x$.
If $x \in \Omega \setminus \{ 0 \}$, then we solve the initial value problems
\begin{equation*}
\dot X^\pm (t) = \frac{b(X^\pm (t))}{\varepsilon} \pm \nu \, \frac{DH(X^\pm (t))}{|DH(X^\pm (t))|}
\ \ \ \text{ and } \ \ \ X^\pm (0) = x.
\end{equation*}
These problems have unique solutions $X^\pm (t)$ for $t \geq 0$
as far as $X^\pm (t)$ stay away from the origin.
Since $b(y) \cdot DH(y) = 0$ for all $y \in \overline \Omega$ and, hence,
\begin{equation} \label{growth-H}
\mathrm{\cfrac{d}{dt}} \, H(X^\pm (t)) = \pm \nu |DH(X^\pm (t))| \ \ \ \text{ for all } t > 0,
\end{equation}
we see that if $x \in (\Omega_i \cup c_i (0)) \setminus \{ 0 \}$ and $i \in \{ 1, 3 \}$,
then $X^- (t_x^-) \in \partial_i \Omega$ and $X^- (t) \in \Omega_i$
for all $t \in (0, t_x^-)$ and for some $t_x^- \in (0, \infty)$.
Similarly, if $x \in (\Omega_2 \cup c_2 (0)) \setminus \{ 0 \}$,
then $X^+ (t_x^+) \in \partial_2 \Omega$ and $X^+ (t) \in \Omega_2$
for all $t \in (0, t_x^+)$ and for some $t_x^+ \in (0, \infty)$.
In view of these observations and the fact that $c_2 (0) = c_1 (0) \cup c_3 (0)$, 
when $x \in \Omega_1 \cup \Omega_3$, we set $\tau^\varepsilon = t_x^-$
and $Y^\varepsilon (t, x) = X^- (t)$ for $t \in [0, \tau^\varepsilon]$,
and when $(\Omega_2 \cup c_2 (0)) \setminus \{ 0 \}$, we set $\tau^\varepsilon = t_x^+$
and $Y^\varepsilon (t, x) = X^+ (t)$ for $t \in [0, \tau^\varepsilon]$.

Now, we consider the case where $x = 0$
and set $\delta = \kappa \wedge (\nu \varepsilon/4)$,
where $\kappa$ is the constant from (H3).
By (H3), for any $y \in B_\delta$, we have $H(y) = y_2^2 -y_1^2$ 
and $|b(y)|/\varepsilon = 2|y|/\varepsilon < \nu/2$.
We set $t_0 = 2\delta/\nu$ and $X(t) = (\nu t/2)(0, 1) \in \mathbb R^2$ for $t \in [0, t_0]$
and note that $X(t_0) \in \Omega_2 \cap \partial B_\delta$
and that $X(t) \in \overline B_\delta$ and $|\dot X(t)| = \nu/2$ for all $t \in [0, t_0]$.
We set $\tau^\varepsilon = \tau^\varepsilon (0) = t_0 + \tau^\varepsilon (X(t_0))$ and 
\begin{align*}
Y^\varepsilon (t, 0) =
\begin{cases}
X(t) \ \ \                                 &\text{ for } t \in [0, t_0], \\
Y^\varepsilon (t - t_0, X(t_0)) \ \ \ &\text{ for } t \in [t_0, \tau^\varepsilon].
\end{cases}
\end{align*}
It is now easily seen that for any $x \in \overline \Omega$,
$Y^\varepsilon (\tau^\varepsilon, x) \in \partial \Omega$ and
$Y^\varepsilon (t, x) \in \overline \Omega$ for all $t \in [0, \tau^\varepsilon]$ and that 
\begin{equation*}
\dot Y^\varepsilon (t, x) - \frac{b(Y^\varepsilon (t, x))}{\varepsilon} \in \overline B_\nu \ \ \ 
\text{ for all } (t, x) \in (0, \tau^\varepsilon) \times \Omega.
\end{equation*}

Observe by the definition of $u^\varepsilon$ and
inequalities \eqref{lower-bound-L} and \eqref{upper-bound-L} that for any $x \in \overline \Omega$,
\begin{equation*}
u^\varepsilon (x) \geq \inf_{\tau \in [0, \infty)} \Big\{ \int_0^\tau -Ce^{-\lambda t} \, dt + (-C)e^{-\lambda \tau} \Big\}
                       = -((C/\lambda) \vee C),
\end{equation*}
and
\begin{equation*}
u^\varepsilon (x) \leq \int_0^{\tau^\varepsilon} L(Y^\varepsilon (t, x), \alpha (t, x))e^{-\lambda t} \, dt + Ce^{-\lambda \tau^\varepsilon}
                      \leq \int_0^{\tau^\varepsilon} Me^{-\lambda t} \, dt + Ce^{-\lambda \tau^\varepsilon} \leq (M/\lambda) \vee C,
\end{equation*}
where
\begin{align*}
\alpha (t, x) :=
\begin{cases}
\dot Y^\varepsilon (t, x) - \cfrac{b(Y^\varepsilon (t, x))}{\varepsilon} \ \ \ &\text{ if } x \in \Omega, \\
0                                                                                 \ \ \ &\text{ if } x \in \partial \Omega. 
\end{cases}
\end{align*}
Thus, we have
\begin{equation*}
|u^\varepsilon (x)| \leq C \vee (M/\lambda) \vee (C/\lambda) \ \ \ 
\text{ for all } (x, \varepsilon) \in \overline \Omega \times (0, \varepsilon_0),
\end{equation*}
which shows that $\{ u^\varepsilon \}_{\varepsilon \in (0, \varepsilon_0)}$ is uniformly bounded on $\overline \Omega$.

Now, we may define the upper (resp., lower) semicontinuous envelope $(u^\varepsilon)^\ast$ (resp., $(u^\varepsilon)_\ast$)
as a bounded function on $\overline \Omega$. 
As is well-known (see, for instance, \cites{I}), $(u^\varepsilon)^\ast$ and $(u^\varepsilon)_\ast$ are, respectively,
a viscosity subsolution and supersolution of 
\begin{equation} \label{pde*}
\lambda u - \cfrac{b \cdot Du}{\varepsilon} + G(x, Du) = 0 \ \ \ \text{ in } \Omega. 
\end{equation}

It remains to prove that $u^\varepsilon \in C(\overline \Omega)$.
We first demonstrate that 
\begin{equation} \label{bdry-c}
\lim_{\overline \Omega \ni y \to x} u^\varepsilon (y) = g^\varepsilon (x) \ \ \ \text{ for all } x \in \partial \Omega,
\end{equation}
where the convergence is uniform in $x \in \partial \Omega$.

To do this, we argue by contradiction
and thus suppose that there exist a sequence $\{ x_n \}_{n \in \mathbb N} \subset \overline \Omega$,
converging to $x_0 \in \partial_i \Omega$ for some $i \in \{ 1, 2, 3 \}$,
and a positive constant $\gamma > 0$ so that
\begin{equation*}
|u^\varepsilon (x_n) - g^\varepsilon (x_0)| \geq \gamma \ \ \ \text{ for all } n \in \mathbb N.
\end{equation*}
There are two cases: for infinitely many $n \in \mathbb N$, we have
\begin{equation} \label{infinite+}
u^\varepsilon (x_n) \geq g^\varepsilon (x_0) + \gamma,
\end{equation}
or, otherwise,
\begin{equation} \label{infinite-}
u^\varepsilon (x_n) \leq g^\varepsilon (x_0) - \gamma.
\end{equation}
By passing to a subsequence, we may assume that, in the first case (resp., in the second case),
\eqref{infinite+} (resp., \eqref{infinite-}) is satisfied for all $n \in \mathbb N$.
We may assume as well that $x_n \in \overline \Omega_i$ for all $n \in \mathbb N$.
The set $H \big(\overline B_\kappa \big) = \{ H(y) \, | \, y \in \overline B_\kappa \}$ is clearly a closed interval,
which we denote by $[h_-, h_+]$,
and, since $\overline B_\kappa \subset \Omega$,
we have $h_1 \vee h_3 < h_- < 0 < h_+ < h_2$.
We may assume by passing once again to a subsequence if necessary that for all $n \in \mathbb N$,
$H(x_n) \in (h_+, h_2]$ if $i = 2$, and $H(x_n) \in [h_i, h_-)$ otherwise.
By \eqref{growth-H}, we see that $Y^\varepsilon (t, x_n) \in \overline \Omega_i \setminus B_\kappa$
for all $t \in [0, \tau^\varepsilon (x_n)]$ and $n \in \mathbb N$.
We set $c_0 = \min_{\overline \Omega \setminus B_\kappa} |DH| (> 0)$.

We treat first the case where \eqref{infinite+} holds for all $n \in \mathbb N$. 
By \eqref{growth-H}, we have
\begin{align*}
(-1)^i h_i - (-1)^i H(x_n) &= (-1)^i H(Y^\varepsilon (\tau^\varepsilon (x_n), x_n)) - (-1)^i H(x_n) \\
                                &= \nu \int_0^{\tau^\varepsilon (x_n)} |DH(Y^\varepsilon (t, x_n))| \, dt \geq \nu c_0 \tau^\varepsilon (x_n),
\end{align*}
and, therefore,
\begin{equation*}
\lim_{n \to \infty} \tau^\varepsilon (x_n) = 0.
\end{equation*}
Setting $\alpha_n (t) = (-1)^i \nu DH(Y^\varepsilon (t, x_n))/|DH(Y^\varepsilon (t, x_n))|$
for $t \in [0, \tau^\varepsilon (x_n)]$, we get
\begin{align*}
u^\varepsilon (x_n) &\leq \int_0^{\tau^\varepsilon (x_n)} L(Y^\varepsilon (t, x_n), \alpha_n (t))e^{-\lambda t} \, dt
                                + g^\varepsilon (Y^\varepsilon (\tau^\varepsilon (x_n), x_n))e^{-\lambda \tau^\varepsilon (x_n)} \\
                         &\leq M \tau^\varepsilon (x_n) + g^\varepsilon (Y^\varepsilon (\tau^\varepsilon (x_n), x_n))e^{-\lambda \tau^\varepsilon (x_n)},       
\end{align*}
and, moreover,
\begin{equation*}
\limsup_{n \to \infty} u^\varepsilon (x_n) \leq g^\varepsilon (x_0),
\end{equation*}
which contradicts \eqref{infinite+}.

Next, we consider the case where \eqref{infinite-} holds for all $n \in \mathbb N$.
We choose $\alpha_n \in L^\infty ([0, \infty); \mathbb R^2)$ and $\tau_n \in [0, \infty)$ for each $n \in \mathbb N$
so that $X_n (t) := X^\varepsilon (t, x_n, \alpha_n) \in \overline \Omega$ for all $t \in [0, \tau_n]$,
$X_n (\tau_n) \in \partial \Omega$, and
\begin{equation} \label{opposite}
u^\varepsilon (x_n) + \frac{\gamma}{2} > \int_0^{\tau_n} L(X_n (t), \alpha_n (t))e^{-\lambda t} \, dt 
                                                     + g^\varepsilon (X_n (\tau_n))e^{-\lambda \tau_n}.
\end{equation}

We define $Z^\varepsilon (t, x)$ for $(t, x) \in [0, \infty) \times \overline \Omega_i \setminus B_\kappa$
as the unique solution of
\begin{equation*}
\dot X(t) = - \, \frac{b(X(t))}{\varepsilon} + (-1)^i \nu \, \frac{DH(X(t))}{|DH(X(t))|} \ \ \ \text{ and } \ \ \ X(0) = x.
\end{equation*}
Similarly to the case of $Y^\varepsilon$, we deduce that there exists $\sigma (x) \in [0, \infty)$ such that
\begin{equation*}
Z^\varepsilon (\sigma (x), x) \in \partial_i \Omega \ \ \ \text{ and } \ \ \
Z^\varepsilon (t, x) \in \overline \Omega_i \setminus B_\kappa \ \ \ \text{ for all } t \in [0, \sigma (x)].
\end{equation*}

We set $s_n = \sigma (x_n)$, $t_n = s_n + \tau_n$, and 
\begin{align*}
Y_n (t) = 
\begin{cases}
Z^\varepsilon (s_n - t, x_n) \ \ \ &\text{ for } t \in [0, s_n], \\
X_n (t - s_n)                   \ \ \ &\text{ for } t \in [s_n, t_n].
\end{cases}
\end{align*}
Note that $Y_n$ is continuous at $t = s_n$ and satisfies
\begin{align*}
\dot Y_n (t) =
\begin{cases}
\cfrac{b(Y_n (t))}{\varepsilon} - (-1)^i \nu \, \cfrac{DH(Y_n (t))}{|DH(Y_n (t))|} \ \ \ &\text{ for } t \in (0, s_n), \\
\cfrac{b(Y_n (t))}{\varepsilon} + \alpha_n (t - s_n) \ \ \ &\text{ for a.e. } t \in (s_n, t_n),
\end{cases}
\end{align*}
that $Y_n (t) \in \overline \Omega$ for all $t \in [0, t_n]$, and that
\begin{equation*}
\lim_{n \to \infty} s_n = 0 \ \ \ \text{ and } \ \ \ \lim_{n \to \infty} Y_n (0) = x_0.
\end{equation*}
Setting
\begin{align*}
\beta_n (t) =
\begin{cases}
- (-1)^i \nu \, \cfrac{DH(Y_n (t))}{|DH(Y_n (t))|} \ \ \ &\text{ for } t \in (0, s_n), \\
\alpha_n (t - s_n)                                      \ \ \ &\text{ for a.e. } t \in (s_n, t_n),
\end{cases}
\end{align*}
we have
\begin{equation*}
\dot Y_n (t) = \frac{b(Y_n (t))}{\varepsilon} + \beta_n (t) \ \ \ \text{ for a.e. } t \in (0, t_n).
\end{equation*}
Since $Y_n (0), Y_n (t_n) \in \partial \Omega$, we see by (G5) that
\begin{equation*}
g^\varepsilon (Y_n (0)) \leq \int_0^{t_n} L(Y_n (t), \beta_n (t))e^{-\lambda t} \, dt
                                     + g^\varepsilon (Y_n (t_n))e^{-\lambda t_n},
\end{equation*}
from which, together with \eqref{opposite} and \eqref{infinite-}, we get
\begin{align*}
g^\varepsilon (Y_n (0)) &\leq Ms_n + e^{-\lambda s_n} \Big( \int_0^{\tau_n} L(Y_n (s_n + t), \alpha_n (t))e^{-\lambda t} \, dt
                           + g^\varepsilon (Y_n (t_n))e^{-\lambda \tau_n} \Big) \\                             
                              &= Ms_n + e^{-\lambda s_n} \Big( \int_0^{\tau_n} L(X_n (t), \alpha_n (t))e^{-\lambda t} \, dt
                           + g^\varepsilon (X_n (\tau_n))e^{-\lambda \tau_n} \Big) \\
                              &< Ms_n + e^{- \lambda s_n} \left( u^\varepsilon (x_n) + \frac{\gamma}{2} \right)
                                \leq Ms_n + e^{-\lambda s_n} \left( g^\varepsilon (x_0) - \frac{\gamma}{2} \right).      
\end{align*}
Sending $n \to \infty$ yields
\begin{equation*}
g^\varepsilon (x_0) \leq g^\varepsilon (x_0) - \frac{\gamma}{2},
\end{equation*}
which is a contradiction.
Thus, we conclude that $u^\epsilon$ satisfies \eqref{bdry-c}.
In particular, we have $u^\varepsilon (x) = g^\varepsilon (x)$ for all 
$x \in \partial \Omega$.

To see the continuity of $u^\varepsilon$, we note that the pde \eqref{pde*} has the form
\begin{equation*}
\lambda u + F(x, Du) = 0 \ \ \ \text{ in } \Omega,
\end{equation*}
where $F$ is given by
\begin{equation*}
F(x, p) = - \, \frac{b(x) \cdot p}{\varepsilon} + G(x, p)
\end{equation*}
and satisfies 
\begin{equation*}
|F(x, p) - F(y, p)| \leq \frac{K}{\varepsilon} |x - y| |p| + m(|x - y|(|p| + 1)) \ \ \ 
\text{ for all } x, y \in \Omega \text{ and } p \in \mathbb R^2,
\end{equation*}
with $K > 0$ being a Lipschitz bound of $b$.
A standard comparison theorem, together with \eqref{bdry-c} and the viscosity properties of $u^\varepsilon$, ensures
that $(u^\varepsilon)^\ast \leq (u^\varepsilon)_\ast$ on $\overline \Omega$,
which implies the continuity of $u^\varepsilon$ on $\overline \Omega$.
\end{proof}

Henceforth, throughout this paper $u^\varepsilon$ denotes the function 
defined in Proposition \ref{viscosity-sol}.

\begin{prop} \label{prop: DPP}
Assume that {\rm (G1)--(G5)} hold.
Let $\varepsilon \in (0, \varepsilon_0)$, $t \geq 0$, and $x \in \overline \Omega$. Then
\begin{align*}
u^\varepsilon (x) = \inf \Big\{ \int_0^{t \wedge \tau^\varepsilon} L(X^\varepsilon (s, x, \alpha )&, \alpha (s)) e^{-\lambda s} \, ds
+ \mathbf{1}_{\{ t < \tau^\varepsilon \}} u^\varepsilon (X^\varepsilon (t, x, \alpha )) e^{-\lambda t} \\ 
&+ \mathbf{1}_{\{ t \geq \tau^\varepsilon \}} g^\varepsilon (X^\varepsilon (\tau^\varepsilon, x, \alpha )) e^{-\lambda \tau^\varepsilon} \ | \ \alpha \in L^\infty (\mathbb R ; \mathbb R^2 ) \Big\},
\end{align*}
where $\tau^\varepsilon$ is a visiting time in $\partial \Omega$
of the trajectory $\{ X^\varepsilon (t, x, \alpha) \}_{t \geq 0}$ constrained in $\overline \Omega$.
\end{prop}

We refer to, for instance, \cites{L} for a proof of this proposition.
The identity in the proposition above is called the \emph{dynamic programming principle}.

We introduce the half relaxed-limits of $u^\varepsilon$ as $\varepsilon \to 0+$:
\begin{align*}
&v^+ (x) = \lim_{r \to 0+} \sup \{ u^\varepsilon (y) \ | \ y \in B_r(x) \cap \overline \Omega, \ \varepsilon \in (0, r) \}, \\
&v^- (x) = \lim_{r \to 0+} \inf \{ u^\varepsilon (y) \ | \ y \in B_r(x) \cap \overline \Omega, \ \varepsilon \in (0, r) \},
\end{align*}
which are well-defined and bounded on $\overline \Omega$
since the family $\{ u^\varepsilon \}_{\varepsilon \in (0, \varepsilon_0)}$
is uniformly bounded on $\overline \Omega$.

In addition to (G1)--(G5), we always assume the following (G6).

\begin{itemize}
\item[(G6)] There exist constants $d_i$, with $i \in \{ 1, 2, 3\}$, such that
				$v^\pm (x) = d_i$ for all $x \in \partial_i \Omega$ and $i \in \{ 1, 2, 3 \}$.
\end{itemize}
Obviously, this implies that
\begin{equation*} \label{g-d_i}
\lim_{\Omega \ni y \to x} v^\pm (y) = \lim_{\varepsilon \to 0+} g^\varepsilon (x) = d_i \ \ \
\text{ uniformly for } x \in \partial_i \Omega \text{ for all } i \in \{ 1, 2, 3 \}.
\end{equation*}

Assumption (G6) is rather implicit and looks restrictive, but it simplifies our arguments below,
since, in the limiting process of sending $\varepsilon \to 0+$,
any boundary layer does not occur. Our use of assumptions (G5) and (G6) 
is somewhat related to the fact that 
\eqref{epHJ} is not coercive when $\varepsilon$ is very small. 
However, for instance, in the case where $G(x, p) = |p| -  f(x)$ with $f \in C(\overline \Omega)$ and $f \geq 0$,
$g^\varepsilon \equiv 0$, and $d_i = 0$ for all $i \in \{ 1, 2, 3 \}$, assumptions (G1)--(G6) hold.
Our formulation of asymptotic analysis of \eqref{epHJ} is based on (G5) and (G6),
which may look a bit silly in the sense that it is not clear which $g^\varepsilon$ and $d_i$, with $i \in \{ 1, 2, 3 \}$, 
satisfy (G5) and (G6).  This question will be taken up in Section 7 and there 
we give a fairly general sufficient condition on the data $(d_1,d_2,d_3)$ 
for which conditions (G5) and (G6) hold.

\section{Main result}

\subsection{The limiting problem}

In this section,
we are concerned with the nonlinear ordinary differential equation 
\begin{equation} \label{lim-HJ}
\lambda u + \overline G_i (h, u') = 0 \ \ \ \text{ in } J_i \text{ and } i\in\{ 1, 2, 3 \}, \tag{$3.1_i$} 
\end{equation} 
where $u : J_i \to \mathbb R$ is the unknown and
$\overline G_i : \bar J_i \setminus \{ 0 \} \times \mathbb R \to \mathbb R$ is the function defined by
\begin{equation*}
\overline G_i (h, q) = \cfrac{1}{T_i (h)} \int_0^{T_i (h)} G \big( X(t, x), qDH(X(t, x)) \big) \, dt,
\end{equation*} 
where $x \in \overline \Omega$ is chosen so that $x \in c_i (h)$.
The value of the integral above is independent of the choice of $x \in c_i (h)$.
In Theorem \ref{main} in the next section, 
the limit of $u^\varepsilon$, as $\varepsilon \to 0+$,
is described by use of an ordered triple of viscosity solutions of \eqref{lim-HJ}, 
with $i \in \{ 1, 2, 3 \}$.

\setcounter{equation}{1}

We give here some lemmas concerning odes \eqref{lim-HJ}.

\begin{lem} \label{G-bar}
For any $i \in \{ 1, 2, 3 \}$, $\overline G_i \in C(\bar J_i \setminus \{ 0 \} \times \mathbb R)$,
and the function $\overline G_i$ is locally coercive in the sense that,
for any compact interval $I$ of $\bar J_i \setminus \{ 0 \}$,
\begin{equation*}
\lim_{r \to \infty} \inf \{ \overline G_i (h, q) \, | \, h \in I, |q| \geq r \} = \infty.
\end{equation*}
\end{lem}

\begin{proof}
The continuity of $\overline G_i$ follows from the definition of $\overline G_i$ and Lemma \ref{T_i}.

Fix any $i \in \{ 1, 2, 3 \}$ and $h_0 \in J_i$.
Set $I = [h_0, h_2]$ if $i = 2$ and, otherwise, $I = [h_i, h_0]$.
We choose $c_0 > 0$ so that
\begin{equation*}
|DH(x)| \geq c_0 \ \ \ \text{ for all } x \in \bigcup_{r \in I} c_i (r).
\end{equation*}
Let $h \in I$ and choose $x \in c_i (h)$.
By \eqref{lower-bound-L}, we get
\begin{align} \label{loc-coer}
\begin{aligned}
\overline G_i (h, q) &\geq \frac{1}{T_i (h)} \int_0^{T_i (h)} \Big( \nu |q| |DH(X(t, x))| - M \Big) \, dt \\
                         &\geq \nu c_0 |q| - M.
\end{aligned}
\end{align}
This shows the local coercivity of $\overline G_i$.
\end{proof}

For $i \in \{1, 2, 3 \}$, let $\mathcal S_i$ (resp., $\mathcal S_i^-$ or $\mathcal S_i^+$) be the set of
all viscosity solutions (resp., viscosity subsolutions or viscosity supersolutions ) of \eqref{lim-HJ}.

\begin{lem} \label{con-ext}
Let $i \in \{ 1, 2, 3 \}$ and $u \in \mathcal S_i^-$.
Then $u$ is uniformly continuous in $J_i$ and, hence,
it can be extended uniquely to $\bar J_i$ as a continuous function on $\bar J_i$.
\end{lem}

\begin{proof}
For any compact interval $I$ of $J_i$,
noting that $u$ is upper semicontinuous in $J_i$ and $\overline G_i$ is locally coercive,
we find that $|u'| \leq C_1 (u, I)$ in the viscosity sense
for some constant $C_1 (u, I) > 0$ depending on $u$ and $I$, 
which shows that $u$ is locally Lipschitz continuous in $J_i$.

Let $h \in J_i$ and fix $x \in c_i (h)$.
By \eqref{loc-coer}, we have
\begin{equation*}
\overline G_i (h, q) \geq \frac{ \nu L_i (h)}{T_i (h)} |q| - M.
\end{equation*}
Hence, we get
\begin{equation} \label{uc}
\lambda u + \frac{\nu L_i}{T_i} |u'| - M \leq 0 \ \ \ \text{ in } J_i
\end{equation}
in the viscosity sense and hence in the almost everywhere sense.

We define $v \in C(J_i)$ by $v(h) = \lambda u(h) - M$ and observe that
\begin{equation*}
|v' (h)| + \frac{\lambda T_i (h)}{\nu L_i (h)} v(h) \leq 0 \ \ \ \text{ for a.e. } h \in J_i.
\end{equation*}
It is obvious that the length $L_i (h)$ of $c_i (h)$ is bounded from below by a positive constant,
while Lemma \ref{T_i} assures that $T_i \in L^1 (J_i)$.
Consequently, we find that $T_i/L_i \in L^1 (J_i)$.
Gronwall's inequality yields, for any $h, a \in J_i$,
\begin{equation} \label{bound-v}
|v(h)| \leq |v(a)| \exp \int_{J_i} \frac{\lambda T_i (s)}{\nu L_i (s)} \, ds,
\end{equation}
which shows that $u$ is a bounded function in $J_i$.
From \eqref{uc}, we get
\begin{equation} \label{abs-con}
|u' (h)| \leq \frac{T_i (h)}{\nu L_i (h)} \left( M + \lambda \sup_{J_i} |u| \right) \ \ \ \text{ for a.e. } h \in J_i.
\end{equation}
Since $T_i/L_i \in L^1 (J_i)$, the inequality above shows that $u$ is uniformly continuous in $J_i$.
\end{proof}

Thanks to the lemma above, we may assume
any $u \in \mathcal S_i^-$, with $i \in \{ 1, 2, 3 \}$, as a function in $C(\bar J_i)$.
To make this explicit notationally, we write $\mathcal S_i^- \cap C(\bar J_i)$ for $\mathcal S_i^-$.
This comment also applies to $\mathcal S_i$ since $\mathcal S_i \subset \mathcal S_i^-$.

The following lemma is a direct consequence of \eqref{abs-con}.

\begin{lem} \label{equi-con}
Let $i \in \{ 1, 2, 3 \}$ and $\mathcal S \subset \mathcal S_i^-$.
Assume that $\mathcal S$ is uniformly bounded on $\bar J_i$.
Then $\mathcal S$ is equi-continuous on $\bar J_i$.
\end{lem}

\begin{lem} \label{bound-by-bc}
Let $i \in \{ 1, 2, 3 \}$ and $u \in \mathcal S_i^- \cap C(\bar J_i)$.
Then there exists a constant $C > 0$, independent of $u$, such that
\begin{equation*}
|u(h)| \leq C(|u(a)| + 1) \ \ \ \text{ for all } h, a \in \bar J_i.
\end{equation*} 
\end{lem}

\begin{proof}
Set
\begin{equation*}
C_1 = \exp \int_{J_i} \frac{\lambda T_i (s)}{\nu L_i (s)} \, ds,
\end{equation*}
and fix $h, a \in \bar J_i$.
According to \eqref{bound-v}, we have
\begin{equation*}
|\lambda u(h) - M| \leq C_1 |\lambda u(a) -M|,
\end{equation*}
and, hence,
\begin{equation*}
|u(h)| \leq C_1 |u(a)| + 2 \lambda^{-1} C_1 M. \qedhere
\end{equation*}
\end{proof}

\subsection{Main result}

The main result is stated as follows.
Recall that, throughout this paper, (H1)--(H3) and (G1)--(G6) are satisfied
and $u^\varepsilon$ is the unique solution of \eqref{epHJ}.

\begin{thm} \label{main}
There exist functions $u_i \in \mathcal S_i \cap C(\bar J_i)$, with $i \in \{ 1, 2, 3 \}$,
such that $u_1 (0) = u_2 (0) = u_3 (0)$,
\begin{equation*}
u_i (h_i) = d_i \ \ \ \text{ for all } i \in \{ 1, 2, 3 \},
\end{equation*}
and, as $\varepsilon \to 0+$,
\begin{equation*}
u^\varepsilon \to u_i \circ H \ \ \ \text{ uniformly on } \overline \Omega_i \text{ for all } i \in \{ 1, 2, 3 \}.
\end{equation*}
That is, if we define $u_0 \in C(\overline \Omega)$ by
\begin{align*}
u_0 (x) = 
\begin{cases}
u_1 \circ H (x) \ \ \ \text{ if } x \in \overline \Omega_1, \\
u_2 \circ H (x) \ \ \ \text{ if } x \in \overline \Omega_2, \\
u_3 \circ H (x) \ \ \ \text{ if } x \in \overline \Omega_3, 
\end{cases}
\end{align*}
then, as $\varepsilon \to 0+$,
\begin{equation*}
u^\varepsilon \to u_0 \ \ \ \text{ uniformly on } \overline \Omega.
\end{equation*}
\end{thm}

Before giving the proof,
we note that the stability of viscosity solutions yields
\begin{equation*} 
-b \cdot Dv^+ \leq 0 \ \ \ \text{ and } \ \ \ -b \cdot Dv^- \geq 0 \ \ \ \text{ in } \Omega,
\end{equation*}
in the viscosity sense.
These show that $v^+$ and $v^-$ are nondecreasing and nonincreasing
along the flow $\{ X(t, x) \}_{t \in \mathbb R}$, respectively.

Fix $i \in \{ 1, 2, 3 \}$ and $x \in \Omega_i$ and set $h = H(x)$.
The monotonicity of $v^+$ along the flow $\{ X(t, x) \}_{t \in \mathbb R}$ yields, for all $t \in [0, T_i (h)]$,
\begin{equation*}
v^+ (x) = v^+ \big( X(T_i (h), x) \big) \geq v^+ (X(t, x)) \geq v^+ (X(0, x)) = v^+ (x).
\end{equation*}
Hence $v^+$ is constant on the loop $c_i (h)$.
Similarly we see that $v^-$ is also constant on the loop $c_i (h)$.

Thus, for any $h \in J_i$ and $i \in \{ 1, 2, 3 \}$,
the image $v^+ (c_i (h)) := \{ v^+ (x) \, | \, x \in c_i (h) \}$ of $c_i (h)$ by $v^+$
(resp., $v^- (c_i (h)) := \{ v^- (x) \, | \, x \in c_i (h) \}$ of $c_i (h)$ by $v^-$)
consists of a single element. This ensures that the relation
\begin{equation} \label{u_i^pm}
u_i^+ (h) \in v^+ (c_i (h)) \ \ \ \text{ (resp., } u_i^- (h) \in v^- (c_i (h)))
\end{equation}
defines a function $u_i^+$ in $J_i$ (resp., $u_i^-$ in $J_i$).
It is easily seen that $u_i^+$ and $u_i^-$ are, respectively,
upper and lower semicontinuous in $J_i$.

For the proof of Theorem \ref{main}, we need the following three propositions.

\begin{thm} \label{characterize}
For all $i \in \{ 1, 2, 3 \}$, $u_i^+ \in \mathcal S_i^-$ and $u_i^- \in \mathcal S_i^+$.
\end{thm}

With this theorem at hand,
we assume (see Lemma \ref{con-ext}) that $u_i^+ \in C(\bar J_i)$ for all $i \in \{ 1, 2, 3 \}$.
Moreover, by assumption (G6), we have
\begin{equation} \label{bc-u_i^+}
u_i^+ (h_i) = \lim_{J_i \ni h \to h_i} u_i^- (h) = v^{\pm} (x) = d_i
\ \ \ \text{ for all } x \in c_i (h_i) \text{ and } i \in \{ 1, 2, 3 \}.
\end{equation}

\begin{lem} \label{v^+-u_i^+}
We have
\begin{equation*}
v^+ (x) \leq \min _{i \in \{ 1, 2, 3 \}} u_i^+ (0) \ \ \ \text{ for all } x \in c_2 (0).
\end{equation*}
\end{lem}

\begin{lem} \label{v^--d_0} 
Set $d_0 = \min_{i \in \{ 1, 2, 3 \}} u_i^+ (0)$. Then
\begin{equation*}
v^- (x) \geq d_0 \ \ \ \text{ for all } x \in c_2 (0). 
\end{equation*}
\end{lem}

Assuming temporarily
Theorem \ref{characterize}, and Lemmas \ref{v^+-u_i^+} and \ref{v^--d_0},
we continue with the

\begin{proof}[Proof of Theorem \ref{main}]
By Theorem \ref{characterize}, we have
$u_i^+ \in \mathcal S_i^- \cap C(\bar J_i)$ and $u_i^- \in \mathcal S_i^+$ for all $i \in \{ 1, 2, 3 \}$.

By definition of the half-relaxed limits, it is obvious that
\begin{equation*} \label{v^--v^+}
v^- \leq v^+ \ \ \ \text{ on } \overline \Omega,
\end{equation*}
that $v^+$ and $-v^-$ are upper semicontinuous on $\overline \Omega$
and that if $v^+ \leq v^-$, then $v^+ = v^-$ and, as $\varepsilon \to 0+$,
\begin{equation*}
u^\varepsilon \to v^- = v^+ \ \ \ \text{ uniformly on }  \overline \Omega.
\end{equation*}

By Lemmas \ref{v^+-u_i^+} and \ref{v^--d_0}, we have
\begin{equation*}
v^+ (x) \leq d_0 \leq v^- (x) \ \ \ \text{ for all } x \in c_2 (0),
\end{equation*}
where $d_0 := \min_{i \in \{ 1, 2, 3 \}} u_i^+ (0)$.
Moreover, by the semicontinuity properties of $v^\pm$, we get
\begin{equation*}
u_i^+ (0) \leq v^+ (x) \leq d_0 \leq v^- (x) \leq \lim_{J_i \ni h \to 0} u_i^- (h)
\ \ \ \text{ for all } x \in c_2 (0) \text{ and } i \in \{ 1, 2, 3 \}.
\end{equation*}
This implies that 
\begin{equation*}
u_i^+ (0) = v^+ (x) = v^- (x) = \lim_{J_i \ni h \to 0} u_i^- (h)
\ \ \ \text{ for all } x \in c_2 (0) \text{ and } i \in \{ 1, 2, 3 \}.
\end{equation*}
According to \eqref{bc-u_i^+}, we have
\begin{equation*}
u_i^+ (h_i) = \lim_{J_i \ni h \to h_i} u_i^- (h) \ \ \ \text{ for all } i \in \{ 1, 2, 3 \}.
\end{equation*}
Thus, by the comparison principle applied to \eqref{lim-HJ},
we find that $u_i^+ = u_i^-$ in $J_i$ for all $i \in \{ 1, 2, 3 \}$.
In particular, setting $u_i = u_i^+$ on $\bar J_i$ for $i \in \{ 1, 2, 3 \}$ and recalling \eqref{bc-u_i^+},
we see that $u_i \in \mathcal S_i \cap C(\bar J_i)$ for all $i \in \{ 1, 2, 3 \}$,
that $u_1 (0) = u_2 (0) = u_3 (0)$, that $u_i (h_i) = d_i$ for all $i \in \{ 1, 2, 3 \}$, and that
\begin{equation*}
v^+ (x) = v^- (x) = u_i (h) \ \ \ \text{ for all } x \in c_i (h), h \in \bar J_i, \text{ and } i \in \{ 1, 2, 3 \},
\end{equation*}
that is, 
\begin{equation*}
v^+ (x) = v^- (x) = u_i \circ H(x) \ \ \ \text{ for all } x \in \overline \Omega_i \text{ and } i \in \{ 1, 2, 3 \}.
\end{equation*} 
This completes the proof. 
\end{proof}

\section{Proof of Theorem \ref{characterize}}

Before giving the proof of Theorem \ref{characterize},
we introduce the functions $\tau_i$ and $\tilde \tau_i$.

For each $i \in \{ 1, 2, 3 \}$, we fix $p_i \in c_i (0) \setminus \{ 0 \}$,
denote by $Y_i (h)$ the solution of the initial value problem
\begin{align*}
\begin{cases}
Y' (h) = \cfrac{DH(Y(h))}{|DH(Y(h))|^2} \ \ \ \text{ for } h \in \bar J_i \setminus \{ 0 \}, \\
Y(0) = p_i,
\end{cases}
\end{align*}
and set
\begin{equation*}
l_i = \{ Y_i (h) \ | \ h \in \bar J_i \setminus \{ 0 \} \}.
\end{equation*}
It is immediate that
\begin{equation*}
Y_i \in C^1(\bar J_i ; \mathbb R^2) \ \ \ \text{ and } \ \ \ H(Y_i (h)) = h \ \ \ \text{ for all } h \in \bar J_i \text{ and } i \in \{ 1, 2, 3 \}. 
\end{equation*}

For $i \in \{ 1, 2, 3 \}$ and $x \in \overline \Omega_i \setminus c_i (0)$,
let $\tau_i (x)$ be the first time
the flow $\{ X(t, x) \}_{t > 0}$ reaches the curve $l_i$, that is,
\begin{equation} \label{def-tau_i}
\tau_i (x) = \inf \{ t > 0 \ | \ X(t, x) \in l_i \}.
\end{equation}  
Note that although $\tau_i$ are continuous
in $\overline \Omega_i \setminus (c_i (0) \cup l_i)$,
they have jump discontinuities across the curves $l_i$.
To avoid this difficulty, for each $i \in \{ 1, 2, 3 \}$,
we modify $\tau_i$ near $l_i$ by considering
the set $U_i = \{ x \in \overline \Omega_i \setminus c_i (0) \ | \ \tau_i (x) \not= T_i \circ H(x)/2 \}$
and the function $\tilde \tau_i : U_i \to (0, \infty)$ defined by
\begin{equation*}
\tilde \tau_i (x) =
\begin{cases}
\tau_i (x)                      \ \ \ &\text{ if } \tau_i (x) > T_i \circ H(x)/2, \\
\tau_i (x) + T_i \circ H(x) \ \ \ &\text{ if } \tau_i (x) < T_i \circ H(x)/2.
\end{cases}
\end{equation*}

\begin{lem}
For all $i \in \{ 1, 2, 3 \}$, $\tau_i \in C^1 \left( \overline \Omega_i \setminus (c_i (0) \cup l_i) \right)$ and $\tilde \tau_i \in C^1 (U_i)$.
\end{lem}

The lemma above, as well as Lemma \ref{T_i},
can be found in \cites{IS} in a slightly different context,
to which we refer for the proof.

\begin{proof}[Proof of Theorem \ref{characterize}]
We only show that $u_1^+ \in \mathcal S_1^-$
since the other cases can be treated in a similar way.

Let $\phi \in C^1 (J_1)$ and $\hat h \in J_1$ and assume that
$\hat h$ is a strict maximum point in $J_1$ of the function $u_1^+ - \phi$.
Set $V_r = \{ x \in \Omega_1 \ | \ |H(x) - \hat h| < r \}$ for $r > 0$
and fix $r > 0$ so that $\overline V_r \subset \Omega_1$.

Fix any $\eta > 0$.
Define the function $g \in C(\Omega_1)$ by
\begin{equation*}
g(x) = G\big( x, \phi' \circ H(x)DH(x) \big),
\end{equation*}
and choose a function $f \in C^1 (\Omega_1)$ so that
\begin{equation*}
|g(x) - f(x)| < \frac{\eta}{2} \ \ \ \text{ for all } x \in V_r.
\end{equation*}

Let $\psi$ be the function in $\Omega_1$ defined by
\begin{equation*}
\psi (x) = \int_0^{\tau_1 (x)} \Big( f(X(t, x)) - \bar f(x) \Big) \, dt,
\end{equation*}
where
\begin{equation*}
\bar f (x) := \cfrac{1}{T_1 \circ H(x)} \int_0^{T_1 \circ H(x)} f(X(s, x)) \, ds,
\end{equation*}
and observe that
\begin{equation} \label{solvability}
\int_0^{T_1 \circ H(x)} \Big( f(X(t, x)) - \bar f(x) \Big) \, dt = 0 \ \ \ \text{ for all } x \in \Omega_1.
\end{equation}

Recalling that
$X \in C^1 (\mathbb R \times \mathbb R^2)$,
$\tau_1 \in C^1 \left( \overline \Omega_1 \setminus (c_1 (0) \cup l_1) \right)$, and $T_1 \in C^1 (\bar J_1 \setminus \{ 0 \})$,
it is clear that $\psi \in C^1 (\Omega_1 \setminus l_1)$.
Moreover, recalling the definition of $\tilde \tau_1$,
we obtain from \eqref{solvability} that
\begin{equation*}
\psi (x) = \int_0^{\tilde \tau_1 (x)} \Big( f(X(t, x)) - \bar f(x) \Big) \, dt \ \ \ \text{ for all } x \in U_1,
\end{equation*}
and, hence, we see that $\psi \in C^1 (U_1)$
and, moreover, that $\psi \in C^1 (\Omega_1)$.
By using the dynamic programming principle, we see that
\begin{equation} \label{eq: colector}
- b \cdot D\psi = -f + \bar f \ \ \ \text{ in } \Omega_1. 
\end{equation} 
Indeed, for any $x \in \Omega_1$ and $s \in \mathbb R$, we have 
\begin{align*}
\begin{aligned}
\psi (X(-s, x)) &= \int_0^{\tau_1 (X(-s, x))} \Big( f \big( X(t, X(-s, x)) \big) - \bar f(X(-s, x)) \Big) \, dt \\
                   &=\int_{-s}^{\tau_1(x)} \Big( f(X(t, x)) - \bar f(x) \Big) \, dt.
\end{aligned}
\end{align*}
Differentiating this with respect to $s$ at $s=0$, we get 
\begin{equation*}
-b(x) \cdot D\psi (x) = -f(x) + \bar f(x). 
\end{equation*}

Choose sequences $\{ \varepsilon_n \}_{n \in \mathbb N} \subset (0, 1)$, converging to zero,
and $\{ y_n \}_{n \in \mathbb N} \subset \overline V_r$ so that $\lim_{n \to \infty} y_n = x_0$
and $\lim_{n \to \infty} u^{\varepsilon_n} (y_n) = v^+ (x_0)$ for some $x_0 \in c_1 (\hat h)$.
Let $\{ x_n \}_{n \in \mathbb N} \subset \overline V_r$ be a sequence consisting of
maximum points over $\overline V_r$ of the functions $u^{\varepsilon_n} - \phi \circ H - \varepsilon_n \psi$.
By replacing the sequence by its subsequence if necessary, we may assume that
$\lim_{n \to \infty} x_n = \hat x$ and $\lim_{n \to \infty} u^{\varepsilon_n} (x_n) = a$
for some $\hat x \in \overline V_r$ and $a \in \mathbb R$.
Noting that, for all $n \in \mathbb N$,
\begin{equation*}
(u^{\varepsilon_n} - \phi \circ H - \varepsilon_n \psi )(y_n) \leq (u^{\varepsilon_n} - \phi \circ H - \varepsilon_n \psi )(x_n),
\end{equation*}
and letting $n \to \infty$, we obtain
\begin{align*}
(u_1^+ - \phi )(\hat h) &= (v^+ - \phi \circ H)(x_0) \leq a - \phi \circ H (\hat x) \\
                              &\leq (v^+ - \phi \circ H)(\hat x) = (u_1^+ - \phi) \circ H(\hat x).
\end{align*}
Since $\hat h$ is a strict maximum point in $J_1$ of the function $u_1^+ - \phi$,
we see that $\hat x \in c_1 (\hat h)$ and, moreover, that $a = u_1^+ (\hat h)$.

If $\varepsilon = \varepsilon_n$ and $n$ is sufficiently large, then $x_n \in V_r$ and
\begin{equation*}
\lambda u^\varepsilon (x_n) - b(x_n) \cdot D\psi (x_n) + G\big( x_n, \phi' \circ H(x_n)DH(x_n) + \varepsilon D\psi (x_n) \big) \leq 0.
\end{equation*} 
Combining this with \eqref{eq: colector} yields
\begin{equation*}
\lambda u^\varepsilon (x_n) - f(x_n) + \bar f(x_n) + G\big( x_n, \phi' \circ H(x_n)DH(x_n) + \varepsilon D\psi (x_n) \big) \leq 0.
\end{equation*} 
Taking the limit, as $n \to \infty$, in the inequality above, we get
\begin{equation*}
\lambda u_1^+ (\hat h) - f(\hat x) + \bar f(\hat x) + g(\hat x) \leq 0,
\end{equation*} 
and, moreover,
\begin{align*}
0 &\geq \lambda u_1^+ (\hat h) - f(\hat x) + g(\hat x) + \frac{1}{T_1 (\hat h)} \int_0^{T_1(\hat h)} f(X(t, \hat x)) \, dt \\
   &> \lambda u_1^+ (\hat h) - \frac{\eta}{2} + \frac{1}{T_1 (\hat h)} \int_0^{T_1(\hat h)} \Big( g(X(t, \hat x)) - \frac{\eta}{2} \Big) \, dt \\ 
   &= \lambda u_1^+ (\hat h) + \overline G_1 (\hat h, \phi' (\hat h)) - \eta.
\end{align*} 
Since $\eta > 0$ is arbitrary, we conclude from the inequality above that
\begin{equation*}
\lambda u_1^+ (\hat h) + \overline G_1 (\hat h, \phi' (\hat h)) \leq 0,
\end{equation*}
and, therefore, $u_1^+ \in \mathcal S_1^-$.
\end{proof}

\section{Proof of lemma \ref{v^+-u_i^+}}

The key point of the proof of Lemma \ref{v^+-u_i^+}
is the behavior of $H(X^\varepsilon (t, x, \alpha))$ regarding the initial value $x$
in a neighborhood of the homoclinic orbit $\{ x \in \mathbb R^2 \ | \ H(x) = 0 \}$.

We write $\mathcal C$ for the subspace of those $\beta \in C^1 (\Omega \setminus \{ 0 \}; \mathbb R^2)$
that are bounded in $\Omega \setminus \{ 0 \}$
and, for each $\varepsilon \in (0, \varepsilon_0)$ and $\beta \in \mathcal C$,
consider the initial value problem
\begin{align} \label{beta}
\begin{cases}
\dot X^\varepsilon (t) = \cfrac{1}{\varepsilon} \, b(X^\varepsilon (t)) + \beta (X^\varepsilon (t)) \ \ \ \text{ for } t \in \mathbb R, \\
X^\varepsilon (0) = x \in \Omega \setminus \{ 0 \}.
\end{cases}
\end{align}
As is well-known, problem \eqref{beta} has a unique solution $X^\varepsilon (t)$,
which is also denoted by $\xi^\varepsilon (t, x, \beta)$,
in the maximal interval  $(\sigma_-^\varepsilon (x, \beta), \sigma_+^\varepsilon (x, \beta))$
where $\sigma_-^\varepsilon (x, \beta) < 0 < \sigma_+^\varepsilon (x, \beta)$,
and the maximality means that
either $\sigma_-^\varepsilon (x, \beta) = -\infty$ or
$\lim_{t \to \sigma_-^\varepsilon (x, \beta) + 0} \mathrm{dist} (X^\varepsilon (t), \partial \Omega \cup \{ 0 \}) = 0$,
and either $\sigma_+^\varepsilon (x, \beta) = \infty$ or
$\lim_{t \to \sigma_+^\varepsilon (x, \beta) - 0} \mathrm{dist} (X^\varepsilon (t), \partial \Omega \cup \{ 0 \}) = 0$.

Next we define $\gamma \in \mathcal C$ by
\begin{equation*}
\gamma (x) = \mu \, \cfrac{DH(x)}{|DH(x)|},
\end{equation*}
where $\mu$ is a positive constant chosen so that
\begin{equation*}
L(x, \xi ) \leq C \ \ \ \text{ for all } (x, \xi ) \in \overline \Omega \times \overline B_\mu \text{ and some } C > 0,
\end{equation*}
and set $h_0 = \min_{i \in \{ 1, 2, 3 \}} |h_i|$ and, for $h \in (0, h_0)$,
\begin{equation} \label{Omega-h}
\Omega_i (h) = \{ x \in \overline \Omega_i \ | \ 0 \leq |H(x)| < h \} \ \text{ for } i \in \{ 1, 2, 3 \} \ \ \text{ and } \ \ \Omega (h) = \bigcup_{i = 1}^3 \Omega_i (h).
\end{equation}

Now, by (H3), we have
\begin{equation*}
|DH(x)|^2 = 4|x|^2 \geq 4|H(x)| \ \ \ \text{ for all } x \in \overline B_\kappa.
\end{equation*}
Since $DH (x) \not= 0$ for all $x \in \Omega \setminus \overline B_\kappa$,
there exists $c_0 \in (0, 2)$ such that
\begin{equation*}
|DH(x)| \geq c_0 |H(x)|^{\frac{1}{2}} \ \ \ \text{ for all } x \in \Omega \setminus \overline B_\kappa.
\end{equation*}
Combining these yields
\begin{equation} \label{ineq: DH}
|DH(x)| \geq c_0 |H(x)|^{\frac{1}{2}} \ \ \ \text{ for all } x \in \Omega.
\end{equation}

\begin{lem} \label{lem: tau_2-tau_1} 
Let $\varepsilon \in (0, \varepsilon_0)$, $h \in (0, h_0)$, and $x \in \Omega (h)$.
If $\tau_1, \tau_2 \in (\sigma_-^\varepsilon (x, \gamma), \sigma_+^\varepsilon (x, \gamma))$ are such that
$\tau_1 < \tau_2$ and $\xi^\varepsilon (t, x, \gamma) \in \Omega (h)$ for all $t \in (\tau_1, \tau_2)$, then
\begin{equation} \label{ineq: tau_2-tau_1}
\tau_2 - \tau_1 \leq \cfrac{2 \sqrt{h}}{c_0 \mu}.
\end{equation}
Also the inequality \eqref{ineq: tau_2-tau_1} holds with $\gamma$ being replaced by $-\gamma$.
\end{lem}

\begin{proof}
Let $\tau_1, \tau_2 \in (\sigma_-^\varepsilon (x, \gamma), \sigma_+^\varepsilon (x, \gamma))$ are such that
$\tau_1 < \tau_2$ and $\xi^\varepsilon (t, x, \gamma) \in \Omega (h)$ for all $t \in (\tau_1, \tau_2)$.
Setting $\psi (r) = r|r|^{-\frac{1}{2}}$ for $r \in \mathbb R \setminus \{ 0 \}$,
we compute that $\psi' (r) = \frac{1}{2} |r|^{-\frac{1}{2}}$
and, moreover, by using \eqref{ineq: DH}, that
\begin{align*}
2 \sqrt{h} &\geq \psi \circ H (\xi^\varepsilon (\tau_2, x, \gamma )) - \psi \circ H (\xi^\varepsilon (\tau_1, x, \gamma )) \\
             &= \int_{\tau_1}^{\tau_2} \psi' \circ H (\xi^\varepsilon (s, x, \gamma )) DH(\xi^\varepsilon (s, x, \gamma )) \cdot \dot \xi^\varepsilon (s, x, \gamma ) \, ds \\
             &=\frac{\mu}{2} \int_{\tau_1}^{\tau_2} |H(\xi^\varepsilon (s, x, \gamma ))|^{-\frac{1}{2}} |DH(\xi^\varepsilon (s, x, \gamma ))| \, ds \geq c_0 \mu (\tau_2 - \tau_1 ),
\end{align*}
from which we conclude that
\begin{equation*}
\tau_2 - \tau_1 \leq \cfrac{2 \sqrt{h}}{c_0 \mu}.
\end{equation*}

Similarly if $\tau_1, \tau_2 \in (\sigma_-^\varepsilon (x, -\gamma), \sigma_+^\varepsilon (x, -\gamma))$ are such that
$\tau_1 < \tau_2$ and $\xi^\varepsilon (t, x, -\gamma) \in \Omega (h)$ for all $t \in (\tau_1, \tau_2)$, then
\begin{align*}
-2 \sqrt{h} &\leq \psi \circ H (\xi^\varepsilon (\tau_2, x, -\gamma )) - \psi \circ H (\xi^\varepsilon (\tau_1, x, -\gamma )) \\
               &= \int_{\tau_1}^{\tau_2} \psi' \circ H (\xi^\varepsilon (s, x, -\gamma )) DH(\xi^\varepsilon (s, x, -\gamma )) \cdot \dot \xi^\varepsilon (s, x, -\gamma ) \, ds \\
               &= - \, \frac{\mu}{2} \int_{\tau_1}^{\tau_2} |H(\xi^\varepsilon (s, x, -\gamma ))|^{-\frac{1}{2}} |DH(\xi^\varepsilon (s, x, -\gamma ))| \, ds \leq -c_0 \mu (\tau_2 - \tau_1 ),
\end{align*}
and, hence,
\begin{equation*}
\tau_2 - \tau_1 \leq \cfrac{2 \sqrt{h}}{c_0 \mu}. \qedhere
\end{equation*}
\end{proof}

For $\varepsilon \in (0, \varepsilon_0)$, $\beta \in \mathcal C$,
$x \in \Omega \setminus \{ 0 \}$, and $t \in (\sigma_-^\varepsilon (x, \beta), \sigma_+^\varepsilon (x, \beta))$,
we define $\pi^\varepsilon [\beta] (t, x) \in \mathbb R^2$ by
\begin{equation*}
\pi^\varepsilon [\beta ] (t, x) = \beta (\xi^\varepsilon (t, x, \beta )).
\end{equation*}
It is clear that $t \mapsto \pi^\varepsilon [\beta] (t, x)$ is a $C^1$ function
in $(\sigma_-^\varepsilon (x, \beta), \sigma_+^\varepsilon (x, \beta))$ and that,
for all $t \in (\sigma_-^\varepsilon (x, \beta), \sigma_+^\varepsilon (x, \beta))$,
\begin{equation*}
X^\varepsilon (t, x, \pi^\varepsilon [\beta ] (\cdot, x)) = \xi^\varepsilon (t, x, \beta ).
\end{equation*}

\begin{proof}[Proof of Lemma \ref{v^+-u_i^+}] 
We may assume that
\begin{equation*}
\sup \{ u^\varepsilon (x) \ | \ x \in \Omega, \ \varepsilon \in (0, \varepsilon_0 ) \} 
\leq C.
\end{equation*}

Let $h \in (0, h_0)$.
Fix any $\eta > 0$ and fix $\delta \in (0, \varepsilon_0)$ so that
if $\varepsilon \in (0, \delta)$, then
\begin{equation*}
u^\varepsilon (x) < v^+ (x) + \eta \ \ \ \text{ for all } x \in \overline{\Omega (h)}.
\end{equation*}

Fix any $\varepsilon \in (0, \delta)$.
Let $x \in \Omega (h)$.
Observe that, for all $t \in (\sigma_-^\varepsilon (x, \gamma), \sigma_+^\varepsilon (x, \gamma))$,
\begin{equation*}
\mathrm{\cfrac{d}{dt}} \, H(\xi^\varepsilon (t, x, \gamma )) = DH(\xi^\varepsilon (t, x, \gamma )) \cdot \gamma (\xi^\varepsilon (t, x, \gamma )) = |DH(\xi^\varepsilon (t, x, \gamma ))| > 0.
\end{equation*}
This shows that $H(\xi^\varepsilon (t, x, \gamma))$ is increasing
in $t \in (\sigma_-^\varepsilon (x, \gamma), \sigma_+^\varepsilon (x, \gamma))$.
It follows from this monotonicity that if $x \in \Omega_2 (h) \setminus \{ 0 \}$,
then $\xi^\varepsilon (\cdot, x, \gamma)$ reaches the loop $c_2 (h)$
at $\tau_2^\varepsilon (x, h) \in (0, \sigma_+^\varepsilon (x, \gamma))$.
Similarly we see that $H(\xi^\varepsilon (t, x, -\gamma))$ is decreasing
in $t \in (\sigma_-^\varepsilon (x, -\gamma)$, $\sigma_+^\varepsilon (x, -\gamma))$,
and, if $x \in \Omega_i (h) \setminus \{ 0 \}$ and $i \in \{ 1, 3 \}$,
then $\xi^\varepsilon (\cdot, x, -\gamma)$ reaches the loop $c_i (-h)$
at $\tau_i^\varepsilon (x, h) \in (\sigma_+^\varepsilon (x, -\gamma), 0)$ (see Fig. 3).

In the case where $x \in \Omega_2 (h) \setminus \{ 0 \}$,
setting $\tau_2 = \tau_2^\varepsilon (x, h) > 0$ and $x_2 = \xi^\varepsilon (\tau_2, x, \gamma) \in c_2 (h)$,  
we have, by Lemma \ref{lem: tau_2-tau_1},
\begin{equation*}
\tau_2 \leq \cfrac{2 \sqrt{h}}{c_0 \mu},
\end{equation*}
and, moreover, by Proposition \ref{prop: DPP},
\begin{align} \label{ineq: x_2}
\begin{aligned}
u^\varepsilon (x) &\leq \int_0^{\tau_2} L \big( X^\varepsilon (s, x, \pi^\varepsilon [\gamma ] (\cdot, x)), \pi^\varepsilon [\gamma ] (\cdot, x) \big) e^{-\lambda s} \, ds  \\
                      &\qquad \qquad \qquad \qquad \qquad \qquad + u^\varepsilon \big( X^\varepsilon (\tau_2, x, \pi^\varepsilon [\gamma ] (\cdot, x)) \big) e^{-\lambda \tau_2} \\ 
                      &\leq \int_0^{\tau_2} L \big( \xi^\varepsilon (s, x, \gamma ), \gamma (\xi^\varepsilon (s, x, \gamma )) \big) e^{-\lambda s} \, ds + u^\varepsilon (x_2) e^{- \lambda \tau_2} \\
                      &\leq C \tau_2 + u^\varepsilon (x_2) + (1 - e^{- \lambda \tau_2}) |u^\varepsilon (x_2)| \\
                      &\leq C(1 + \lambda ) \tau_2 + u^\varepsilon (x_2) \leq C(1 + \lambda ) \, \cfrac{2 \sqrt{h}}{c_0 \mu} + u^\varepsilon (x_2).
\end{aligned}             
\end{align}

\begin{figure}[t]
\centering
\input{trajectory.tex}
\caption{\ }
\end{figure}

Similarly, in the case where $x \in \Omega_i (h) \setminus \{ 0 \}$ and $i \in \{ 1, 3 \}$,
setting $\tau_i = \tau_i^\varepsilon (x, h) > 0$ and $x_i = \xi^\varepsilon (\tau_i, x, -\gamma) \in c_i (-h)$, we have
\begin{equation*}
\tau_i \leq \cfrac{2 \sqrt{h}}{c_0 \mu},
\end{equation*}
and
\begin{align} \label{ineq: x_i}
\begin{aligned}
u^\varepsilon (x) &\leq \int_0^{\tau_i} L \big( \xi^\varepsilon (s, x, -\gamma ), -\gamma (\xi^\varepsilon (s, x, -\gamma )) \big) e^{-\lambda s} \, ds + u^\varepsilon (x_i) e^{-\lambda \tau_i} \\
					  &\leq C(1 + \lambda ) \tau_i + u^\varepsilon (x_i) \\
					  &\leq C(1 + \lambda ) \cfrac{2 \sqrt{h}}{c_0 \mu} + u^\varepsilon (x_i).
\end{aligned}             
\end{align}

From the monotonicity of $H(\xi^\varepsilon (t,  x, \gamma))$
in $t \in (\sigma_-^\varepsilon (x, \gamma), \sigma_+^\varepsilon (x, \gamma))$,
it also follows that if $x \in c_i (0) \setminus \{ 0 \}$ and $i \in \{ 1, 3 \}$,
then $\xi^\varepsilon (\cdot, x, \gamma)$ reaches the loop $c_2 (h)$
at $\tau_{i, +}^\varepsilon (x, h, \gamma) \in (0, \sigma_+^\varepsilon (x, \gamma))$
and reaches the loop $c_i (-h)$ at $\tau_{i, -}^\varepsilon (x, h, \gamma) \in (\sigma_-^\varepsilon (x, \gamma), 0)$.
Similarly if $x \in c_i (0) \setminus \{ 0 \}$ and $i \in \{ 1, 3 \}$,
then $\xi^\varepsilon (\cdot, x, -\gamma)$ reaches the loop $c_i (-h)$
at $\tau_{i, +}^\varepsilon (x, h, -\gamma) \in (0, \sigma_+^\varepsilon (x, -\gamma))$
and reaches the loop $c_2 (h)$ at $\tau_{i, -}^\varepsilon (x, h, -\gamma) \in (\sigma_-^\varepsilon (x, -\gamma), 0)$.

Now, let $x = p_1$ (recall here that $p_i$ are the fixed points on $c_i (0) \setminus \{ 0 \}$). Setting
\begin{align*}
\begin{cases}
s_1= -\tau_{1, -}^\varepsilon (p_1, h, \gamma ), \ s_2 = \tau_{1, +}^\varepsilon (p_1, h, \gamma ), \ 
s_3 = -\tau_{1, -}^\varepsilon (p_1, h, -\gamma ), \ s_4 = \tau_{1, +}^\varepsilon (p_1, h, -\gamma ), \\
y_1 = \xi^\varepsilon (-s_1, p_1, \gamma ) \in c_1 (-h), \ y_2 = \xi^\varepsilon (s_2, p_1, \gamma ) \in c_2 (h), \\
y_3 = \xi^\varepsilon (-s_3, p_1, -\gamma ) \in c_2 (h), \ y_4 = \xi^\varepsilon (s_4, p_1, -\gamma ) \in c_1 (-h) \text{ (see Fig. 4), }
\end{cases}
\end{align*}
we see that
\begin{equation*}
\xi^\varepsilon (t, y_1, \gamma ) = \xi^\varepsilon (t, \xi^\varepsilon (-s_1, p_1, \gamma ), \gamma ) = \xi^\varepsilon (t - s_1, p_1, \gamma ) \in \Omega_1 (h) \ \ \ \text{ for all } t \in (0, s_1),
\end{equation*}
and
\begin{equation*}
\xi^\varepsilon (s_1, y_1, \gamma ) = \xi^\varepsilon (0, p_1, \gamma ) = p_1.
\end{equation*}
Similarly we see that
\begin{equation*}
\xi^\varepsilon (t, y_1, \gamma ) = \xi^\varepsilon (t - s_1, p_1, \gamma ) \in \Omega_2 (h) \ \ \ \text{ for all } t \in (s_1, s_1 + s_2),
\end{equation*}
and
\begin{equation*}
\xi^\varepsilon (s_1 + s_2, y_1, \gamma ) = y_2.
\end{equation*}
Moreover we have
\begin{equation*}
s_k \leq \cfrac{2 \sqrt{h}}{c_0 \mu} \ \ \ \text{ for all } k \in \{ 1, 2, 3, 4 \}.
\end{equation*}
From these, we get 
\begin{align} \label{ineq: y_2}
\begin{aligned}
u^\varepsilon (y_1) &\leq \int_0^{s_1 + s_2} L \big( \xi^\varepsilon (s, y_1, \gamma ), \gamma (\xi^\varepsilon (s, x, \gamma )) \big) e^{-\lambda s} \, ds + u^\varepsilon (y_2) e^{-\lambda (s_1 + s_2)} \\
                         &\leq C(1 + \lambda )(s_1 + s_2) + u^\varepsilon (y_2) \\
                         &\leq C(1 + \lambda ) \cfrac{4 \sqrt{h}}{c_0 \mu} + u^\varepsilon (y_2).
\end{aligned}
\end{align}
Similarly we get
\begin{align} \label{ineq: y_4}
\begin{aligned}
u^\varepsilon (y_3) &\leq \int_0^{s_3 + s_4} L \big( \xi^\varepsilon (s, y_1, -\gamma ), -\gamma (\xi^\varepsilon (s, x, -\gamma )) \big) e^{-\lambda s} \, ds + u^\varepsilon (y_4) e^{-\lambda (s_3 + s_4)} \\
                         &\leq C(1 + \lambda ) \cfrac{4 \sqrt{h}}{c_0 \mu} + u^\varepsilon (y_4).
\end{aligned}
\end{align}

\begin{figure}[t]
\centering
\input{trajectory-1.tex}
\caption{\ }
\end{figure}

Next, let $x = p_3$. Setting
\begin{align*}
\begin{cases}
t_1= -\tau_{1, -}^\varepsilon (p_3, h, \gamma ), \ t_2 = \tau_{1, +}^\varepsilon (p_3, h, \gamma ), \ 
t_3 = -\tau_{1, -}^\varepsilon (p_3, h, -\gamma ), \ t_4 = \tau_{1, +}^\varepsilon (p_3, h, -\gamma ), \\
z_1 = \xi^\varepsilon (-t_1, p_3, \gamma ) \in c_3 (-h), \ z_2 = \xi^\varepsilon (t_2, p_3, \gamma ) \in c_2 (h), \\
z_3 = \xi^\varepsilon (-t_3, p_3, -\gamma ) \in c_2 (h), \ z_4 = \xi^\varepsilon (t_4, p_3, -\gamma ) \in c_3 (-h),
\end{cases}
\end{align*}
we have
\begin{equation*}
t_k \leq \cfrac{2 \sqrt{h}}{c_0 \mu} \ \ \ \text{ for all } k \in \{ 1, 2, 3, 4 \},
\end{equation*}
and, moreover,
\begin{align*}
\begin{aligned}
u^\varepsilon (z_1) &\leq \int_0^{t_1 + t_2} L \big( \xi^\varepsilon (s, z_1, \gamma ), \gamma (\xi^\varepsilon (s, x, \gamma )) \big) e^{-\lambda s} \, ds + u^\varepsilon (z_2) e^{-\lambda (t_1 + t_2)} \\
                         &\leq C(1 + \lambda ) \cfrac{4 \sqrt{h}}{c_0 \mu} + u^\varepsilon (z_2),
\end{aligned}
\end{align*}
and
\begin{align*}
\begin{aligned}
u^\varepsilon (z_3) &\leq \int_0^{t_3 + t_4} L \big( \xi^\varepsilon (s, z_1, -\gamma ), -\gamma (\xi^\varepsilon (s, x, -\gamma )) \big) e^{-\lambda s} \, ds + u^\varepsilon (z_4) e^{-\lambda (t_3 + t_4)} \\
                         &\leq C(1 + \lambda ) \cfrac{4 \sqrt{h}}{c_0 \mu} + u^\varepsilon (z_4).
\end{aligned}
\end{align*}

For $i \in \{ 1, 2, 3 \}$ and $x, y \in c_i ((-1)^i h)$, we set
\begin{equation*}
t_i (x, y) = \inf \{ t \geq 0 \ | \ X^\varepsilon (t, x) = y \}.
\end{equation*}

Fix $i \in \{ 1, 2, 3 \}$.
As noted before, if $x \in c_i ((-1)^i h)$,
then the solution $X(t, x)$ of \eqref{HS} is periodic with period $T_i ((-1)^i h)$.
This periodicity is readily transferred to the solution $X^\varepsilon (t, x)$ of \eqref{epHS},
thanks to the scaling property $X(\varepsilon^{-1} t, x) = X^\varepsilon (t, x)$,
and $X^\varepsilon (t, x)$ is periodic with period $T_i^\varepsilon ((-1)^i h) := \varepsilon T_i ((-1)^i h)$.
Hence we see that
\begin{equation*}
0 \leq t_i (x, y) < \varepsilon T_i ((-1)^i h) \ \ \ \text{ for all } x, y \in c_i ((-1)^i h).
\end{equation*}

Let $x \in \Omega_2 (h) \setminus \{ 0 \}$
and let $\tau_2 > 0$ and $x_2 \in c_2 (h)$ are as in \eqref{ineq: x_2}.
By using \eqref{ineq: x_2}, we get
\begin{equation} \label{ineq: u_2}
u^\varepsilon (x) \leq C(1 + \lambda ) \cfrac{2 \sqrt{h}}{c_0 \mu} + v^+ (x_2) + \eta = C(1 + \lambda ) \cfrac{2 \sqrt{h}}{c_0 \mu} + u_2^+ (h) + \eta.
\end{equation}
Next, noting that $x_2, y_3 \in c_2 (h)$, we have
\begin{align*}
u^\varepsilon (x_2) &\leq \int_0^{t_2 (x_2, y_3)} L(X^\varepsilon (s, x_2), 0) e^{- \lambda s} \, ds + u^\varepsilon (y_3) e^{- \lambda t_2 (x_2, y_3)} \\
                         &\leq C(1 + \lambda ) \varepsilon T_2 (h) + u^\varepsilon (y_3).
\end{align*}
Combining this with \eqref{ineq: x_2} and \eqref{ineq: y_4} yields
\begin{equation*}
u^\varepsilon (x) \leq C(1 + \lambda ) \left( \cfrac{6 \sqrt{h}}{c_0 \mu} + \varepsilon T_2 (h) \right) + u^\varepsilon (y_4),
\end{equation*}
and, hence, noting that $y_4 \in c_1 (-h)$, we get
\begin{equation} \label{ineq: u_1}
u^\varepsilon (x) \leq C(1 + \lambda ) \left( \cfrac{6 \sqrt{h}}{c_0 \mu} + \varepsilon T_2 (h) \right) + u_1^+ (-h) + \eta.
\end{equation}
Similarly we get
\begin{equation*}
u^\varepsilon (x_2) \leq C(1 + \lambda ) t_2 (x_2, z_3) + u^\varepsilon (z_3),
\end{equation*}
\begin{equation*}
u^\varepsilon (x) \leq C(1 +\lambda ) \left( \cfrac{6 \sqrt{h}}{c_0 \mu} + \varepsilon T_2 (h) \right) + u^\varepsilon (z_4),
\end{equation*}
and
\begin{equation} \label{ineq: u_3}
u^\varepsilon (x) \leq C(1 + \lambda ) \left( \cfrac{6 \sqrt{h}}{c_0 \mu} + \varepsilon T_2 (h) \right) + u_3^+ (-h) + \eta.
\end{equation}
Now, \eqref{ineq: u_2}, \eqref{ineq: u_1}, and \eqref{ineq: u_3} together yield
\begin{equation} \label{ineq: sup-Omega_2}
\sup_{\overline{\Omega_2 (h)}} u^\varepsilon \leq C(1 + \lambda ) \left( \cfrac{6 \sqrt{h}}{c_0 \mu} + \varepsilon T_2 (h) \right) + \min_{i \in \{ 1, 2, 3 \}} u_i^+ ((-1)^i h) + \eta.
\end{equation}
Here we have used the fact that
\begin{equation*}
\sup_{\Omega_2 (h) \setminus \{ 0 \}} u^\varepsilon = \sup_{\overline{\Omega_2 (h)}} u^\varepsilon,
\end{equation*}
which is a consequence of the continuity of $u^\varepsilon$ on $\overline \Omega$.

Next, let $x \in \Omega_1 (h) \setminus \{ 0 \}$
and let $\tau_1 > 0$ and $x_1 \in c_1 (-h)$ are as in \eqref{ineq: x_i}.
Noting that $x_1, y_1 \in c_1(-h)$, we have
\begin{align*}
u^\varepsilon (x_1) &\leq \int_0^{t_1 (x_1, y_1)} L(X^\varepsilon (s, x_1), 0) e^{- \lambda s} \, ds + u^\varepsilon (y_1) e^{- \lambda t_1 (x_1, y_1)} \\
                         &\leq C(1 + \lambda )\varepsilon T_1 (-h) + u^\varepsilon (y_1).
\end{align*}
Combining this with \eqref{ineq: x_i} and \eqref{ineq: y_2} yields
\begin{equation*}
u^\varepsilon (x) \leq C(1 + \lambda ) \left( \cfrac{6 \sqrt{h}}{c_0 \mu} + \varepsilon T_1 (-h) \right) + u^\varepsilon (y_2),
\end{equation*}
and, hence,
noting that $y_2 \in c_2 (h) \subset \overline \Omega_2 (h)$ and combining this with \eqref{ineq: sup-Omega_2},
we get
\begin{equation*}
u^\varepsilon (x) \leq C(1 + \lambda ) \left( \cfrac{12 \sqrt{h}}{c_0 \mu} + \varepsilon (T_1 (-h) + T_2 (h)) \right) + \min_{i \in \{ 1, 2, 3 \}} u_i^+ ((-1)^i h) + \eta,
\end{equation*}
from which we conclude that
\begin{equation} \label{ineq: sup-Omega_1}
\sup_{\overline{\Omega_1 (h)}} u^\varepsilon \leq C(1 + \lambda ) \left( \cfrac{12 \sqrt{h}}{c_0 \mu} + \varepsilon (T_1 (-h) + T_2 (h)) \right) + \min_{i \in \{ 1, 2, 3 \}} u_i^+ ((-1)^i h) + \eta.
\end{equation}

An argument parallel to the above yields
\begin{equation*}
\sup_{\overline{\Omega_3 (h)}} u^\varepsilon \leq C(1 + \lambda ) \left( \cfrac{12 \sqrt{h}}{c_0 \mu} + \varepsilon (T_3 (-h) + T_2 (h)) \right) + \min_{i \in \{ 1, 2, 3 \}} u_i^+ ((-1)^i h) + \eta,
\end{equation*}
and, moreover, combining this
with \eqref{ineq: sup-Omega_2} and \eqref{ineq: sup-Omega_1} gives us the inequality
\begin{equation*}
\sup_{\overline{\Omega (h)}} u^\varepsilon \leq C(1 + \lambda ) \left( \cfrac{12 \sqrt{h}}{c_0 \mu} + \varepsilon (T_1 (-h) + T_2 (h) + T_3 (-h)) \right) + \min_{i \in \{ 1, 2, 3 \}} u_i^+ ((-1)^i h) + \eta.
\end{equation*}
Since $\eta > 0$ and $\varepsilon \in (0, \delta)$ are arbitrary,
we conclude from the inequality above that
\begin{equation*}
\sup_{c_2 (0)} v^+ \leq \sup_{\overline{\Omega (h)}} v^+ \leq C(1 + \lambda ) \cfrac{8 \sqrt{h}}{c_0 \mu} + \min_{i \in \{ 1, 2, 3 \}} u_i^+ ((-1)^i h),
\end{equation*}
and, therefore,
\begin{equation*}
\sup_{c_2 (0)} v^+ \leq \lim_{h \to 0+} \min_{i \in \{ 1, 2, 3 \}} u_i^+ ((-1)^i h) = \min_{i \in \{ 1, 2, 3 \}} u_i^+ (0). \qedhere
\end{equation*}
\end{proof}

\section{Proof of Lemma \ref{v^--d_0}}

We recall here that $h_0 = \min_{i \in \{ 1, 2, 3 \}} |h_i|$ and
that $\Omega (\cdot)$ denotes
the set defined in \eqref{Omega-h}.

The key step of the proof of Lemma \ref{v^--d_0} is

\begin{lem} \label{key}
For any $\eta > 0$,
there exist $\delta \in (0, h_0)$ and $\psi \in C^1 (\Omega (\delta))$ such that
\begin{equation*}
-b \cdot D\psi + G(x, 0) < G(0, 0) + \eta \ \ \ \text{ in } \Omega (\delta ).
\end{equation*}
\end{lem}

For $r \in (0, \kappa)$ and $x \in \Omega (r^2) \setminus \overline B_r$,
let $\tau_+^r (x)$ be the first time the flow $\{ X(t, x) \}_{t > 0}$ reaches $\partial B_r$, that is,
\begin{equation*}
\tau_+^r (x) = \inf \{ t > 0 \ | \ X(t, x) \in \partial B_r \},
\end{equation*}
and, similarly, set
\begin{equation*}
\tau_-^r (x) = \sup \{ t < 0 \ | \ X(t, x) \in \partial B_r \},
\end{equation*}
and 
\begin{equation*}
T^r (x) = (\tau_+^r - \tau_-^r) (x) ( > 0).
\end{equation*}

\begin{lem} \label{low-bound-T_i^r}
We have
\begin{equation} \label{ineq-T_i^r}
T^r (x) \geq \log \frac{\kappa}{r} - \frac{\log 2}{2},
\end{equation}
for all $x \in \Omega (r^2) \setminus \overline B_r$ and $r \in (0, \kappa)$.
\end{lem}

\begin{proof}
Fix any $r \in (0, \kappa)$.
If $r > \kappa/\sqrt{2}$, then the right-hand side of the inequality \eqref{ineq-T_i^r} is negative.

We assume henceforth that $r \leq \kappa/\sqrt{2}$.
Set
\begin{equation*}
E (x) = \{ t \in [-\tau_-^r (x), \tau_+^r (x)] \ | \ X(t, x) \in \overline B_{\kappa} \setminus B_r \} \ \ \ \text{ for } x \in \Omega (r^2) \setminus \overline B_r, 
\end{equation*}
note that 
\begin{equation*}
T^r (x) \geq |E (x)| \ \ \ \text{ for all } x \in \Omega (r^2) \setminus \overline B_r,
\end{equation*}
and fix any $x \in \Omega (r^2) \setminus \overline B_r$.

Consider first the case where $x \in c_2 (0)$.
For each $s \in [r, \kappa]$,
the line $c_2 (0) \cap \left( \overline B_s \setminus B_r \right) \cap \{ (x_1, x_2) \, | \, x_1 \geq 0, \, x_2 \geq 0 \}$
is represented as
\begin{equation} \label{line}
X \left( t, \frac{1}{\sqrt{2}} (r, r) \right) = \frac{r}{\sqrt{2}} \, (e^{2t}, e^{2t}) \ \ \ \text{ with } t \in \left[0, \, \frac{1}{2}\log\frac{s}{r} \right],
\end{equation}
which implies that
\begin{equation*} 
|E(x)| = \log\frac{\kappa}{r}.
\end{equation*}
Hence, we have
\begin{equation} \label{T^r-bound1}
T^r (x) \geq \log\frac{\kappa}{r}.
\end{equation}

Consider next the case where $x \not \in c_2 (0)$.
Set $h = |H(x)|$.
For each $s \in [r, \kappa]$,
the curve (hyperbola) $c_2 (h) \cap \overline B_s \cap \{ (x_1, x_2) \, | \, x_2 > 0 \}$ is represented as
\begin{equation} \label{curve}
X \left(t, (0, \sqrt{h}) \right) = \frac{\sqrt{h}}{2} (e^{2t} -e^{-2t}, e^{2t} + e^{-2t}) \ \ \ \text{ with } t \in [-\sigma^s (h), \sigma^s (h)],
\end{equation} 
where
\begin{equation*} \label{sigma^s}
\sigma^s (h) := \frac{1}{4} \log \left( \frac{s^2}{h} + \sqrt{\frac{s^4}{h^2} - 1} \right).
\end{equation*}

Noting, by the rotational symmetry, that
\begin{equation*} 
|E(x)| = 2(\sigma^{\kappa} (h) - \sigma^r (h)),
\end{equation*}
we compute
\begin{align*}
T^r (x) &\geq 2(\sigma^{\kappa} (h) - \sigma^r (h)) 
           \geq \frac{1}{2} \log \frac{\kappa^2}{h} - \frac{1}{2} \log \frac{2r^2}{h} \\
         &= \log \frac{\kappa}{r} - \frac{\log 2}{2}.
\end{align*}
Combining this with \eqref{T^r-bound1} completes the proof.
\end{proof}

\begin{lem} \label{tau^r}
For all $r \in (0, \kappa)$,
$\tau_{\pm}^r \in C^1 \left( \Omega (r^2) \setminus \overline B_r \right)$.
\end{lem}

\begin{proof}
Fix any $r \in (0, \kappa)$.
Let $\varphi \in C^1 (\mathbb R^2)$ and $F \in C^1 (\mathbb R \times \mathbb R^2)$ are the functions defined by
\begin{equation*}
\varphi (x_1, x_2) = x_1^2 + x_2^2 -r^2 \ \ \ \text{ and } \ \ \ F(t, x) = \varphi (X(t, x)).
\end{equation*}
Observe that for all $x \in \Omega (r^2) \setminus \overline B_r$,
$F(\tau_{\pm}^r (x), x) = 0$, and, in view of (H3),
\begin{align*}
\mathrm{\cfrac{d}{dt}} \, F(t, x) \Big|_{t = \tau_{\pm}^r (x)} = D\varphi \big( X(\tau_{\pm}^r (x), x) \big) \cdot \dot X (\tau_{\pm}^r (x), x)
                                                   = 8 X_1 (\tau_{\pm}^r (x), x) X_2 (\tau_{\pm}^r (x), x) \not= 0,
\end{align*}
where $X := (X_1, X_2)$.
Now the claim follows from the implicit function theorem.
\end{proof}

The following lemma is a consequence of Lemma \ref{tau^r}.

\begin{lem} \label{T^r}
For all $r \in (0, \kappa)$, $T^r \in C^1 \left( \Omega (r^2) \setminus \overline B_r \right)$.
\end{lem}

\begin{proof}[Proof of Lemma \ref{key}]
Fix any $\eta > 0$.
Choose a function $\widetilde G \in C^1 (\overline \Omega)$ so that
\begin{equation*}
|G(x, 0) - \widetilde G (x)| < \frac{\eta}{8} \ \ \ \text{ for all } x \in \overline \Omega,
\end{equation*}
define the function $g \in C^1 (\overline \Omega)$ by
\begin{equation*}
g(x) = \widetilde G(x) - \widetilde G(0),
\end{equation*}
and fix $\delta \in (0, \kappa/2)$ so that
\begin{equation*}
|g(x)| < \frac{\eta}{4} \ \ \ \text{ for all } x \in B_{2\delta}.
\end{equation*}

For each $r > 0$,
we choose a cut-off function $\zeta_r \in C^1 (\mathbb R^2)$ so that
\begin{align*} 
\begin{cases}
\zeta_r (x) = 0             \ \ \ &\text{ if } |x| \leq r, \\
0 \leq \zeta_r (x) \leq 1 \ \ \ &\text{ if } r \leq |x| \leq 2r, \\
\zeta_r (x) = 1             \ \ \ &\text{ if } |x| \geq 2r,
\end{cases}
\end{align*}
and define the function $f \in C^1 (\overline \Omega)$ by
\begin{equation*}
f (x) = g (x) \zeta_{\delta} (x),
\end{equation*}
which satisfies 
\begin{equation*}
f  = 0 \ \ \ \text{ on } \overline B_{\delta} \ \ \ \text{ and } \ \ \ |g - f| < \frac{\eta}{4} \ \ \ \text{ on } \overline \Omega.
\end{equation*}

Fix a small $r \in (0, \delta/4)$.
Let $\psi$ be the function in $\Omega (r^2) \setminus \overline B_r$ defined by
\begin{equation*}
\psi (x) = \int_0^{\tau_+^r (x)} \Big( f(X(t, x)) - \bar f (x) \Big) \, dt,
\end{equation*}
where
\begin{equation*}
\bar f (x) := \cfrac{1}{T^r (x)} \int_{\tau_-^r (x)}^{\tau_+^r (x)} f(X(s, x)) \, ds.
\end{equation*}

Recalling that $X \in C^1 (\mathbb R \times \mathbb R^2)$ and
$\tau_{\pm}^r$, $T^r \in C^1 \left( \Omega (r^2) \setminus \overline B_r \right)$,
it is clear that $\psi \in C^1 \left( \Omega (r^2) \setminus \overline B_r \right)$.
By using the dynamic programming principle,
we see that
\begin{equation} \label{psi_i}
-b \cdot D\psi = -f + \bar f \ \ \ \text{ in } \Omega (r^2) \setminus \overline B_r.
\end{equation}

Next, note that
\begin{equation*}
\lim_{\Omega (r^2) \setminus \overline B_r \ni y \to x} \psi (y) = 0 \ \ \ 
\text{ for all } x \in \Omega (r^2) \cap \partial B_r,
\end{equation*}
and let $\tilde \psi$ and $\varphi$ are the functions in $\Omega (r^2)$ defined by 
\begin{align*}
\tilde \psi (x) = 
\begin{cases}
\psi (x)   \ \ \ &\text{ if } x \in \Omega (r^2) \setminus \overline B_r, \\
0           \ \ \ &\text{ if } x \in \Omega (r^2) \cap \overline B_r, 
\end{cases}
\end{align*}
and
\begin{equation*}
\varphi (x) = \tilde \psi (x) \zeta_{2r} (x).
\end{equation*}
Noting that $\tilde \psi \in C(\Omega (r^2)) \cap C^1 (\Omega (r^2) \setminus \partial B_r)$,
we see that $\varphi \in C^1 (\Omega (r^2))$ and, moreover, that
\begin{equation} \label{varphi}
-b \cdot D \varphi = -b \cdot D \tilde \psi \zeta_{2r} - \tilde \psi b \cdot D\zeta_{2r} \ \ \ \text{ in } \Omega (r^2) \setminus \partial B_r.
\end{equation}

Now it is enough to show that
if $r \in (0, \delta/4)$ is sufficiently small, then $\varphi$ satisfies
\begin{equation} \label{enough}
- b \cdot D \varphi + G(x, 0) < G(0, 0) + \eta \ \ \ \text{ in } \Omega (r^2).
\end{equation}

Since $\varphi (x) = f (x) = 0$ for all $x \in \overline B_{2r}$, we conclude that
\begin{equation} \label{center}
-b(x) \cdot D \varphi (x) + f(x) = 0 \ \ \ \text{ for all } x \in \overline B_{2r}.
\end{equation}

In the case where $x \in \Omega (r^2) \setminus B_{4r}$,
since $\zeta_{2r} (x) = 1$ and $D\zeta_{2r} (x) = 0$,
by \eqref{varphi}, we see that
\begin{equation} \label{bvarphi-bpsi_i}
-b(x) \cdot D\varphi (x) = -b (x) \cdot D \tilde \psi (x) = -b (x) \cdot D\psi (x).
\end{equation}
Fix any $y \in \Omega (r^2) \setminus \overline B_{\delta}$.
Noting that
\begin{equation*}
f(X(t, y)) = 0 \ \ \ \text{ for all } t \in [\tau_-^r (y), \tau_-^{\delta} (y)] \cup [\tau_+^{\delta} (y), \tau_+^r (y)],
\end{equation*}
we have
\begin{equation} \label{T_i^delta} 
\int_{\tau_-^r (z)}^{\tau_+^r (z)} f(X(t, z)) \, dt = \int_{\tau_-^\delta (y)}^{\tau_+^\delta (y)} f (X(t, y)) \, dt \ \ \ 
\text{ for all } z \in \Omega (r^2) \setminus \overline B_r.
\end{equation}
Combining \eqref{psi_i}, \eqref{bvarphi-bpsi_i}, and \eqref{T_i^delta}
and using Lemma \ref{low-bound-T_i^r}, we get
\begin{align*} 
\begin{aligned}
|- b(x) \cdot D\varphi (x) + f (x)| &= |\bar f (x)| \\ 
                                            &\leq \left( \log \frac{\kappa}{r} - \frac{\log 2}{2} \right)^{-1} \int_{\tau_-^\delta (y)}^{\tau_+^\delta (y)} |f (X(t, y))| \, dt,
\end{aligned}
\end{align*}
from which, by replacing $r \in (0, \delta /4)$ by a smaller number if necessary, we conclude that  
\begin{equation} \label{outer}
- b(x) \cdot D\varphi (x) + f (x) < \frac{\eta}{2}.
\end{equation}

Let $x \in \Omega (r^2) \cap \left( B_{4r} \setminus \overline B_{2r} \right)$.
Noting that $f (x) = 0$, by \eqref{varphi}, we have
\begin{equation} \label{varphi-psi_1} 
-b(x) \cdot D\varphi (x) + f(x) = -b(x) \cdot D \tilde \psi (x) \zeta_{2r} (x) - \tilde \psi (x)b(x) \cdot D\zeta_{2r} (x).
\end{equation}
Using \eqref{T_i^delta}, we get
\begin{align*}
- b (x) \cdot D\tilde \psi (x) \zeta_{2r} (x) &\leq |- b (x) \cdot D\tilde \psi (x)| \\
                                                    &= |\bar f (x)| \leq \left( \log \frac{\kappa}{r} - \frac{\log 2}{2} \right)^{-1} \int_{\tau_-^\delta (y)}^{\tau_+^\delta (y)} |f (X(t, y))| \, dt, 
\end{align*}
from which,
by replacing $r \in (0, \delta /4)$ by a smaller number if necessary, we conclude that  
\begin{equation} \label{middle1}
-b (x) \cdot D\tilde \psi (x) \zeta_{2r} (x) < \cfrac{\eta}{4}.
\end{equation}

Next we note, by (H3), that $|b(y)| = O(|y|)$ as $y \to 0$.
We may assume that $|D\zeta_{2r} (y)| \leq C_0 |2r|^{-1}$ for all $y \in B_{4r}$ and for some $C_0 > 0$.
Hence we have
\begin{equation*}
|b (y) \cdot D\zeta_{2r} (y)| \leq C_1 \ \ \ \text{ for all } y \in B_{4r} \text{ and for some } C_1 > 0,
\end{equation*}
and
\begin{equation} \label{zeta_r}
-\tilde \psi b \cdot D\zeta_{2r} \leq C_1 |\tilde \psi| \ \ \ \text{ in } \Omega (r^2) \cap \left( B_{4r} \setminus \overline B_{2r} \right).
\end{equation}

Now, if $x \in c_2 (0)$, then, in view of \eqref{line}, we see that
\begin{equation*}
\tau_+^r (x) \wedge \tau_-^r (x) \leq \frac{1}{2} \log \frac{4r}{r} = \log 2,
\end{equation*}
and, if $x \not \in c_2 (0)$ and $h = |H(x)|$, then, in view of \eqref{curve}, we see that
\begin{equation*}
\tau_+^r (x) \wedge \tau_-^r (x) \leq \sigma^{4r} (h)- \sigma^r (h) \leq \frac{5}{4} \log 2.
\end{equation*}
These together yield
\begin{equation*}
\tau_+^r (x) \wedge \tau_-^r (x) \leq \log 2.
\end{equation*}

Consider first the case where $\tau_+^r (x) \wedge \tau_-^r (x) = \tau_-^r (x)$.
Using \eqref{T_i^delta}, we get
\begin{align} \label{psi_1}
\begin{aligned}
\tilde \psi (x) = \left( 1- \frac{\tau_+^r (x)}{T^r (x)} \right) \int_{\tau_-^\delta (y)}^{\tau_+^\delta (y)} f(X(t, y)) \, dt.
\end{aligned}
\end{align}
Since 
\begin{equation*}
T^r (x) - \log 2 \leq \tau_+^r (x) < T^r (x),
\end{equation*}
by Lemma \ref{low-bound-T_i^r}, we have
\begin{equation*}
0 < 1- \frac{\tau_+^r (x)}{T^r (x)} \leq \frac{\log 2}{T^r (x)} \leq \log 2 \left( \log \cfrac{\kappa}{r} -\cfrac{\log 2}{2} \right)^{-1},
\end{equation*}
and, hence, by replacing $r \in (0, \delta /4)$ by a smaller number if necessary,
we conclude from \eqref{psi_1} that
\begin{equation*}
C_1 |\tilde \psi (x)| < \frac{\eta}{4},
\end{equation*}
and, moreover, from \eqref{zeta_r}, that
\begin{equation} \label{psi_1-b-zeta_r}
-\tilde \psi (x)b(x) \cdot D\zeta_{2r} (x) < \frac{\eta}{4}.
\end{equation}

Consider next the case where $\tau_+^r (x) \wedge \tau_-^r (x) = \tau_+^r (x)$.
Noting that
\begin{equation*}
f (X(t, x)) = 0 \ \ \ \text{ for all } t \in [0, \tau_+^r (x)],
\end{equation*}
and using \eqref{T_i^delta}, we get
\begin{align*}
|\tilde \psi (x)| \leq \frac{\tau_+^r (x)}{T^r (x)} \int_{\tau_-^\delta (y)}^{\tau_+^\delta (y)} |f(X(t, y))| \, dt
                    \leq \log 2 \left( \log \cfrac{\kappa}{r} - \cfrac{\log 2}{2} \right)^{-1} \int_{\tau_-^\delta (y)}^{\tau_+^\delta (y)} |f(X(t, y))| \, dt,
\end{align*}
from which, by replacing $r \in (0, \delta /4)$ by a smaller number if necessary, we conclude that  
\begin{equation*}
C_1 |\tilde \psi (x)| < \frac{\eta}{4},
\end{equation*}
and, moreover, that
\begin{equation*}
-\tilde \psi (x)b(x) \cdot D\zeta_{2r} (x) < \frac{\eta}{4}.
\end{equation*}
Combining this with \eqref{psi_1-b-zeta_r} yields
\begin{equation} \label{middle2}
-\tilde \psi b \cdot D\zeta_{2r} < \frac{\eta}{4} \ \ \ \text{ in } \Omega (r^2) \cap \left( B_{4r} \setminus \overline B_{2r} \right).
\end{equation} 
Now, \eqref{varphi-psi_1}, \eqref{middle1}, and \eqref{middle2} together yield
\begin{equation*}
-b \cdot D\varphi + f < \frac{\eta}{2} \ \ \ \text{ in } \Omega (r^2) \cap \left( B_{4r} \setminus \overline B_{2r} \right).
\end{equation*} 
Combining this with \eqref{center} and \eqref{outer}, we get
\begin{equation*}
-b \cdot D\varphi + f < \frac{\eta}{2} \ \ \ \text{ in } \Omega (r^2),
\end{equation*}
and, moreover, for any $x \in \Omega (r^2)$,
\begin{align*}
\frac{\eta}{2} &> - b (x) \cdot D\varphi (x) + f (x) > - b (x) \cdot D\psi (x) + g (x) - \frac{\eta}{4} \\
                  &= - b (x) \cdot D\varphi (x) + \tilde G(x) - \tilde G(0) -\frac{\eta}{4} \\
                  &> - b (x) \cdot D\varphi (x) + G(x, 0) - G(0, 0) - \frac{\eta}{2}.
\end{align*}
This proves \eqref{enough}. The proof is complete.
\end{proof}

\begin{lem} \label{lim-min-barG_i}
We have 
\begin{equation*}
\lim_{h \to 0+} \min_{q \in \mathbb R} \overline G_i ((-1)^i h, q) = G(0, 0) \ \ \ \text{ for all } i \in \{ 1, 2, 3 \}.
\end{equation*}
\end{lem}

In what follows, let $L(c)$ denote the length of a given curve $c$.

\begin{lem} \label{length-c_i}
For any $r \in (0, \kappa)$,
\begin{equation*}
L(c_i (h) \cap B_r) \leq 4r \ \ \ \text{ for all } h \in J_i \text{ and } i \in \{ 1, 2, 3 \}.
\end{equation*}
\end{lem}

\begin{proof}
Fix any $i \in \{ 1, 2, 3 \}$ and $h \in J_i$.
If $|h| \geq r^2$, then $c_i (h) \cap B_r = \emptyset$ and $L(c_i (h) \cap B_r) = 0$.

We assume henceforth that $|h| < r^2$ and, for the time being, that $i = 2$.
As seen in the proof of Lemma \ref{low-bound-T_i^r},
the curve (hyperbola) $c_2 (h) \cap B_r \cap \{ (x_1, x_2) \, | \, x_2 > 0 \}$ is represented as
\begin{equation*}
(x_1, x_2) = (\xi_1 (t), \xi_2 (t)) \ \ \ \text{ with } t \in (-\tau, \tau),
\end{equation*} 
where
\begin{equation*}
(\xi_1 (t), \xi_2 (t)) := \frac{\sqrt{h}}{2} \, (e^{2t} -e^{-2t}, e^{2t} + e^{-2t}),
\end{equation*}
and
\begin{equation*}
\tau := \frac{1}{4} \log \left( \frac{r^2}{h} + \sqrt{\frac{r^4}{h^2} - 1} \right).
\end{equation*}

We note
\begin{equation*}
e^{2\tau} \leq \sqrt{\frac{2}{h}} \, r,
\end{equation*}
and compute that
\begin{align*}
L(c_2 (h) \cap B_r) &= 2L(c_2 (h) \cap B_r \cap \{ (x_1, x_2) \, | \, x_2 > 0 \})
                           = 4\int_0^{\tau} \sqrt{{\dot \xi_1 (t)}^2 + {\dot \xi_2 (t)}^2} \, dt \\
                         &= 4\sqrt{2h} \int_0^{\tau} \sqrt{e^{4t} + e^{-4t}} \, dt 
                           \leq 4\sqrt{2h} \int_0^{\tau} (e^{2t} + e^{-2t}) \, dt \leq 2\sqrt{2h} e^\tau \leq 4r.
\end{align*}

By the rotational symmetry, the computation above also implies
that $L(c_i (h) \cap B_r) \leq 2r$ when $i = 1$ or $i = 3$.
The proof is complete.
\end{proof}

\begin{proof}[Proof of Lemma \ref{lim-min-barG_i}] 
We begin by noting that, for some $c_0 > 0$,
\begin{equation} \label{c_0}
|DH(x)| \geq c_0 |x| \ \ \ \text{ for all } x \in \overline \Omega.
\end{equation}
Indeed, by (H3), we have
\begin{equation} \label{DH-near-0}
|DH(x)| = 2|x| \ \ \ \text{ for all } x \in B_{\kappa},
\end{equation}
and also, we have
\begin{equation*}
\min_{\overline \Omega \setminus B_{\kappa}} |DH(x)| > 0.
\end{equation*}
These together show that \eqref{c_0} hold for some constant $c_0 > 0$.

Fix any $\gamma > 0$.
In view of the coercivity of $G$, we choose $R > 0$ so that
$G(x, p) \geq G(0, 0)$ for all $(x, p) \in \overline \Omega \times (\mathbb R^n \setminus B_R)$.
Since $G$ is uniformly continuous on $\overline \Omega \times \overline B_R$,
there is a constant $C_1 > 0$ such that
\begin{equation*}
|G(x, p) - G(0, 0)| \leq \gamma + C_1 (|x| + |p|) \ \ \ \text{ for all } (x, p) \in \overline \Omega \times \overline B_R.
\end{equation*}
Hence, we have
\begin{align*}
&G(x, p) \geq G(0, 0) - \gamma -C_1 (|x| + |p|) \ \ \ \text{ for all } (x, p) \in \overline \Omega \times \mathbb R^2, \\
&|G(x, 0) - G(0, 0)| \leq \gamma + C_1|x| \ \ \ \text{ for all } x \in \overline \Omega. 
\end{align*}

The two inequalities above combined with \eqref{c_0} yield
\begin{align}
&G(x, p) \geq G(0, 0) - \gamma -C_2 (|DH(x)| + |p|) \ \ \ \text{ for all } (x, p) \in \overline \Omega \times \mathbb R^2, \label{G-bound1} \\
&|G(x, 0) - G(0, 0)| \leq \gamma + C_2|DH(x)| \ \ \ \text{ for all } x \in \overline \Omega \label{G-bound2}
\end{align}
for some constant $C_2 > 0$.

Fix any $i \in \{ 1, 2, 3 \}$ and $h \in J_i$.
Fix a point $x \in c_i (h)$.
Taking the average of both sides of \eqref{G-bound2} along $c_i (h)$, we get
\begin{equation} \label{G_i-bound0}
|\overline G_i (h, 0) - G(0, 0)| \leq \gamma + \frac{C_2}{T_i (h)} \int_0^{T_i (h)} |DH(X(t, x))| \, dt = \gamma + \frac{C_2 L_i (h)}{T_i (h)}.
\end{equation}

Fix any $q \in \mathbb R$.
Consider first the case where $q \leq \gamma T_i (h)$.
Using \eqref{G-bound1}, we compute
\begin{align} \label{G_i-bound1}
\begin{aligned}
\overline G_i (h, q) - G(0, 0) + \gamma &\geq - \, \frac{C_2}{T_i (h)} \int_0^{T_i (h)} (1 + |q|)|DH(X(t, x))| \, dt \\
                                                   &= - (1 + |q|) \frac{C_2 L_i (h)}{T_i (h)} \geq  - \, \frac{C_2 L_i (h)}{T_i (h)} - \gamma C_2 L_i (h).
\end{aligned}
\end{align}

Consider next the case where $q > \gamma T_i (h)$.
Set
\begin{equation*}
S = \{ t \in [0, T_i (h)] \ | \ |X(t, x)| \leq R/(c_0 |q|) \}. 
\end{equation*}
Choose $\delta > 0$ so that if $|h| \leq \delta$, then
\begin{equation*}
\frac{R}{c_0 \gamma T_i (h)} \leq \kappa,
\end{equation*}
and assume henceforth that $|h| < \delta$.
Note that $R/(c_0 |q|) \leq \kappa$ and, by Lemma \ref{length-c_i},
\begin{equation*}
L(c_i (h) \cap B_{R/(c_0 |q|)}) \leq \frac{4R}{c_0 |q|},
\end{equation*}
which implies that
\begin{equation} \label{G_i-bound2}
\int_S |DH(X(t, x))| \, dt = L(c_i (h) \cap B_{R/(c_0 |q|)}) \leq \frac{4R}{c_0 |q|}.
\end{equation}
On the other hand, if $t \in [0, T_i (h)] \setminus S$, then, by \eqref{c_0},
\begin{equation*}
\frac{R}{c_0 |q|} < |X(t, x)| \leq \frac{|DH(X(t, x))|}{c_0},
\end{equation*}
that is, $|q| |DH(X(t, x))| > R$. Hence,
\begin{equation} \label{G_i-bound3}
G \big( X(t, x), qDH(X(t, x)) \big) \geq G(0, 0) \ \ \ \text{ for all } t \in [0, T_i (h)] \setminus S.
\end{equation}

Using \eqref{G-bound1}, \eqref{G_i-bound2}, and \eqref{G_i-bound3},
we compute that
\begin{align*}
\int_0^{T_i (h)} &G(X(t, x), qDH(X(t, x))) \, dt \\
                   &\geq \int_S \Big( G(0, 0) - \gamma - C_2(1+|q|)|DH(X(t, x))| \Big) \, dt + G(0, 0) \int_{[0, T_i (h)] \setminus S} \, dt \\
                   &\geq (G(0, 0) - \gamma) T_i (h) - C_2 L_i (h) - C_2 |q| L(c_i (h) \cap B_{R/(c_0 |q|)}) \\
                   &\geq (G(0, 0) - \gamma) T_i (h) - C_2 L_i (h) - \frac{4C_2 R}{c_0},
\end{align*}
from which we get
\begin{equation*}
\overline G_i (h, q) \geq G(0, 0) - \gamma - \frac{L_i (h)}{T_i (h)} - \frac{2R}{c_0 T_i (h)} \ \ \ \text{ if } |h| < \delta.
\end{equation*}

The inequality above together with \eqref{G_i-bound1} implies that
\begin{equation*}
\liminf_{J_i \ni h \to 0} \min_{q \in \mathbb R} \overline G_i (h, q) \geq G(0, 0) \ \ \ \text{ for all } i \in \{ 1, 2, 3 \},
\end{equation*}
while \eqref{G_i-bound0} yields
\begin{equation} \label{lim-barG_i-00}
\lim_{J_i \ni h \to 0} \overline G_i (h, 0) = G(0, 0) \ \ \ \text{ for all } i \in \{ 1, 2, 3 \}.
\end{equation}
These two together complete the proof.
\end{proof}

Formula \eqref{lim-barG_i-00} can be easily generalized as
\begin{equation} \label{general-lim-barG_i}
\lim_{J_i \times \mathbb R \ni (h, q) \to (0, 0)} \overline G_i (h, q) = G(0, 0) \ \ \ \text{ for all } i \in \{ 1, 2, 3 \}.
\end{equation}

As in Lemma \ref{v^--d_0}, for $i \in \{ 1, 2, 3 \}$,
let $u_i^+ \in \mathcal S_i^- \cap C(\bar J_i)$ be the function
defined by \eqref{u_i^pm} in $J_i$ and extended by continuity to $\bar J_i$
and let $d_0 = \min_{i \in \{ 1, 2, 3 \}} u_i^+ (0)$.

\begin{lem} \label{u_i^+0-G00}
We have
\begin{equation*}
\lambda u_i^+ (0) + G(0, 0) \leq 0 \ \ \ \text{ for all } i \in \{ 1, 2, 3 \}.
\end{equation*}
\end{lem}

\begin{proof}
Fix $i \in \{1, 2, 3 \}$.
Since $u_i^+ \in \mathcal S_i^- \cap C(\bar J_i)$,
we see that
\begin{equation*}
\lambda u_i^+ (h) + \min_{q \in \mathbb R} \overline G_i (h, q) \leq 0 \ \ \ \text{ for all } h \in J_i.
\end{equation*}
By Lemma \ref{lim-min-barG_i},
we conclude that $\lambda u_i^+ (0) + G(0, 0) \leq 0$.
\end{proof}

\begin{lem} \label{nu_i^d} 
Let $d \in (-\infty, d_0)$ and set
\begin{equation} \label{def-nu_i^d}
\nu_i^d (h) = \sup \{ u (h) \ | \ u \in \mathcal S_i^- \cap C(\bar J_i), \ u(0) = d \} \ \ \ \text{ for } h \in \bar J_i \text{ and } i \in \{ 1, 2, 3 \}.
\end{equation}
Then there exists $\delta \in (0, h_0)$ such that
\begin{equation*}
\nu_i^d (h) > d \ \ \ \text{ for all } h \in J_i \cap [-\delta, \delta ] \text{ and } i \in \{ 1, 2, 3 \}.
\end{equation*}
\end{lem}

Here an important remark on $\mathcal S_i^-$ is
that if $u \in \mathcal S_i^-$, with $i \in \{ 1, 2, 3 \}$,
then $u - a \in \mathcal S_i^-$ for any constant $a > 0$.
In particular, the sets of all $u \in \mathcal S_i^-$ that satisfy $u(0) = d$, with $d < d_0$ and $i \in \{ 1, 2, 3 \}$,
are non-empty and, by Lemmas \ref{equi-con} and \ref{bound-by-bc},
these are uniformly bounded and equi-continuous on $\bar J_i$.
Thus, the functions $\nu_i^d$, with $i \in \{ 1, 2, 3 \}$, are well-defined as continuous functions on $\bar J_i$ and,
in view of Perron's method, these are solutions of \eqref{lim-HJ}.

\begin{proof} 
Since $d < d_0$, we have
\begin{equation*}
\lambda d + G(0, 0) < 0 \ \ \ \text{ and } \ \ \ \min_{i \in \{ 1, 2, 3 \}} u_i^+ (0) > d.
\end{equation*}
By \eqref{general-lim-barG_i}, there exists $\delta \in (0, h_0)$ such that,
for all $i \in \{ 1, 2, 3 \}$ and $h \in J_i \cap [ -\delta, \delta ]$,
\begin{equation*}
\lambda (d + (-1)^i h \delta ) + \overline G_i (h, (-1)^i \delta ) < 0 \ \ \ \text{ and } \ \ \ u_i^+ (h) \geq d + (-1)^i h \delta.
\end{equation*}

For $i \in \{ 1, 2, 3 \}$, we set 
\begin{align*}
w_i (h) = 
\begin{cases}
u_i^+ (h) - u_i^+ ((-1)^i \delta ) + d + \delta^2 \ \ \ &\text{ if } h \in \bar J_i \setminus [-\delta, \delta ], \\
d + (-1)^i h \delta                                       \ \ \ &\text{ if } h \in \bar J_i \cap [-\delta, \delta ], 
\end{cases}
\end{align*}
and observe that $w_i \in \mathrm{Lip} (\bar J_i)$ and 
\begin{equation*}
\lambda w_i (h) + \overline G_i (h, Dw_i (h)) \leq 0 \ \ \ \text{ for a.e. } h \in J_i.
\end{equation*}
This and the convexity of $\overline G_i (h, \cdot)$ imply that $w_i \in \mathcal S_i^-$ for all $i \in \{ 1, 2, 3 \}$.
Noting that $w_i (0) = d$, we see that $\nu_i^d \geq w_i$ on $\bar J_i$ for all $i \in \{ 1, 2, 3 \}$, which,
in particular, shows that $\nu_i^d (h) \geq d + (-1)^i h \delta > d$ for all $h \in J_i \cap [-\delta, \delta]$ and $i \in \{ 1, 2, 3 \}$.
The proof is complete.
\end{proof}

\begin{proof}[Proof of Lemma \ref{v^--d_0}] 
We argue by contradiction.
We thus set $d = \min_{c_2 (0)} v^-$ and suppose that $d < d_0$.

For $i\in\{1,2,3\}$, let $\nu_i^d$ be the functions defined by \eqref{def-nu_i^d}.
For $i \in \{ 1, 2, 3 \}$ and $h \in \bar J_i$, we set
\begin{equation*}
v_i (h) = u_i^+ (h) \wedge \nu_i^d (h),
\end{equation*}
and observe that $v_i \in \mathcal S_i \cap C(\bar J_i)$, $v_i (0) = d$, and $v_i (h_i) \leq u_i^+ (h_i) = d_i$. 
Noting that $u_i^-\in\mathcal S_i^+$, $\lim_{J_i\ni h\to h_i} u_i^- (h) =d_i$, 
and $\lim_{J_i\ni h\to 0}u_i^-(h)\geq d$ for all $i\in\{1,2,3\}$ and using the comparison principle, we get
\begin{equation*}
u_i^-(h) \geq v_i (h) \ \ \ \text{ for all } h \in J_i \text{ and } i \in \{ 1, 2, 3 \}.
\end{equation*}

Since $d < d_0$, thanks to Lemma \ref{u_i^+0-G00},
we may choose $\eta > 0$ so that
\begin{equation*}
\lambda d + G(0, 0) < - \eta.
\end{equation*}
By Lemma \ref{key},
there exist $\delta \in (0, h_0)$ and $\psi \in C^1 (\overline{\Omega (\delta)})$ such that
\begin{equation*}
-b \cdot D\psi + G(x, 0) \leq G(0, 0) + \eta \ \ \ \text{ on } \overline{\Omega (\delta )}.
\end{equation*}
In view of Lemma \ref{nu_i^d},
by replacing $\delta \in (0, h_0)$ by a smaller number if necessary, we may assume that
\begin{equation*}
v_i (h) > d \ \ \ \text{ for all } h \in J_i \cap [-\delta, \delta ] \text{ and } i \in \{ 1, 2, 3 \}. 
\end{equation*}
Now we may choose $c \in (d, d_0)$ so that
\begin{equation*}
c < \min_{i \in \{ 1, 2, 3 \}} v_i ((-1)^i \delta ) \ \ \ \text{ and } \ \ \ \lambda c + G(0, 0) < -\eta.
\end{equation*}

Fix any $\gamma > 0$ so that $\gamma < c - d$.
That is, $c - \gamma > d$. Set
\begin{equation*}
w^\varepsilon (x) = c - \gamma + \varepsilon \psi (x) \ \ \ \text{ for } x \in \overline{\Omega (\delta )} \text{ and } \varepsilon \in (0, \varepsilon_0 ),
\end{equation*}
and compute that, for any $\varepsilon \in (0, \varepsilon_0)$ and $x \in \overline{\Omega (\delta)}$,
\begin{align} \label{ineq: w^ep}
\begin{aligned}
\lambda w^\varepsilon (x) &- \cfrac{b (x) \cdot Dw^\varepsilon (x)}{\varepsilon} + G(x, Dw^\varepsilon (x)) \\
                                  &= \lambda (c - \gamma ) + \varepsilon \lambda \psi (x) - b(x) \cdot D\psi (x) + G(x, \varepsilon D\psi (x)) \\
                                  &= \lambda (c - \gamma ) + \varepsilon \lambda \psi (x) - b(x) \cdot D \psi (x) + G(x, 0) + G(x, \varepsilon D\psi (x)) - G(x, 0) \\
                                  &\leq - \lambda \gamma + \varepsilon \lambda \psi (x) + \lambda c + G(0, 0) + \eta + G(x, \varepsilon D \psi (x)) - G(x, 0) \\
                                  &< -\lambda \gamma + \varepsilon \lambda \psi (x) + G(x, \varepsilon D\psi (x)) - G(x, 0),
\end{aligned}
\end{align}
from which, by replacing $\varepsilon_0 \in (0, 1)$ by a smaller number if necessary,
we may assume that if $\varepsilon \in (0, \varepsilon_0)$, then 
\begin{equation*}
\lambda w^\varepsilon (x) - \cfrac{b(x) \cdot Dw^\varepsilon (x)}{\varepsilon} + G(x, Dw^\varepsilon (x)) \leq 0 \ \ \ \text{ for all } x \in \Omega (\delta ).
\end{equation*}
Since
\begin{equation*}
c - \gamma < \min_{i \in \{ 1, 2, 3 \}} v_i ((-1)^i \delta) - \gamma,
\end{equation*}
by replacing $\varepsilon_0 \in (0, 1)$ by a smaller number if necessary,
we may assume that if $\varepsilon \in (0, \varepsilon_0)$, then
\begin{equation*}
w^\varepsilon (x) < \min_{i \in \{ 1, 2, 3 \}} v_i ((-1)^i \delta ) -\gamma \ \ \ \text{ for all } x \in \overline{\Omega (\delta )}.
\end{equation*}
Moreover, since $u_i^- \geq v_i$ in $J_i$,
by replacing $\varepsilon_0 \in (0, 1)$ by a smaller number if necessary,
we may assume that if $\varepsilon \in (0, \varepsilon_0)$, then
\begin{equation*}
\min_{i \in \{ 1, 2, 3 \}} v_i ((-1)^i \delta ) - \gamma < u^\varepsilon (x) \ \ \ \text{ for all } x \in c_i ((-1)^i \delta ) \text{ and } i \in \{ 1, 2 ,3 \}.
\end{equation*}
Noting that
\begin{equation*}
\partial \Omega (\delta ) = \bigcup_{i \in \{ 1, 2, 3 \}} c_i ((-1)^i \delta ),
\end{equation*}
and applying the comparison principle on $\overline{\Omega (\delta)}$,
we get
\begin{equation*}
w^\varepsilon(x) \leq u^\varepsilon(x) \ \ \ \text{ for all } x \in \overline{\Omega (\delta )} \text{ and } \varepsilon \in (0, \varepsilon_0),
\end{equation*}
which yields 
\begin{equation*}
c - \gamma \leq v^- (x) \ \ \ \text{ for all } x \in c_2 (0).
\end{equation*}
Since $d<c-\gamma$, this is a contradiction.
\end{proof}

\section{The boundary data for the odes}

In this section, we \textit{do not} assume (G5) and (G6).

Given $(d_0, d_1, d_2, d_3) \in \mathbb R^4$,
we consider here the admissibility of the boundary data $(d_0, d_1, d_2, d_3)$
for the boundary value problem, with $i \in \{ 1, 2, 3 \}$,
\begin{align}\label{HJ}
\begin{cases}
\lambda u_i + \overline G_i (h, u_i') = 0 \ \ \ \text{ in } J_i, \\ \tag{$\mathrm{HJ}$}
u_i (h_i) = d_i, \\
u_i (0) = d_0,
\end{cases}
\end{align}
where the admissibility means that, with given $(d_1,d_2,d_3)$, 
conditions (G5) and (G6) hold 
for some boundary data $g^\varepsilon$.

For $i \in \{1, 2, 3 \}$,
let $I_i$ be the set of $d \in \mathbb R$ such that the set
\begin{equation*}
\{ u \in \mathcal S_i^- \cap C(\bar J_i) \ | \ u(h_i) = d \}
\end{equation*}
is nonempty.

Fix $i \in \{ 1, 2, 3 \}$.
Since $\lambda > 0$, it is obvious that
if $d \in I_i$ and $c < d$, then $c \in I_i$.
Observe that if $d \in \mathbb R$ satisfies
\begin{equation*}
\lambda d + \max_{h \in \bar J_i} \overline G_i (h, 0) \leq 0, 
\end{equation*}
then $d \in \mathcal S_i^-$ and $d \in I_i$,
and that if $d \in \mathbb R$ satisfies
\begin{equation*}
\lambda d + \min_{(h, p) \in \bar J_i \times \mathbb R} \overline G_i (h, p) > 0,
\end{equation*} 
then $d \not \in I_i$.
Thus we see that $I_i = (-\infty, a_i]$
for some $a_i \in \mathbb R$.

For $i \in \{ 1, 2, 3 \}$, $d \in I_i$, and $h \in \bar J_i$, we set
\begin{equation*}
\rho_i^d (h) = \sup \{ u(h) \ | \ u \in \mathcal S_i^- \cap C(\bar J_i), \ u(h_i) = d \},
\end{equation*}
and observe that $\rho_i^d \in \mathcal S_i\cap C(\bar J_i)$ and $\rho_i^d (h_i)=d$.

We set
\begin{equation*}
\rho_0 = \min_{i \in \{ 1, 2, 3 \}} \sup_{d \in I_i} \rho_i^d (0),
\end{equation*}
and note, in view of \eqref{bound-by-bc}, that $\rho_0 < \infty$.

Let $I_0$ be the set of $d \in \mathbb R$ such that
\begin{equation*}
\{ u \in \mathcal S_i^- \cap C(\bar J_i) \ | \ u(0) = d \} \not= \emptyset \ \ \ \text{ for all } i \in \{ 1, 2, 3 \}.
\end{equation*}
It is obvious, as well as $I_i$, that if $d \in I_0$ and $c < d$, then $c \in I_0$.
Observe that if $d > \rho_0$, then $d \not\in I_0$,
and that if $d \in I_0$, 
then 
there exist $c_i \in I_i$, with $i \in \{ 1, 2, 3 \}$,
such that $\rho_i^{c_i} (0) = d$.
Thus we see that $I_0 = (-\infty, \rho_0]$.

For $i \in \{ 1, 2, 3 \}$, $d \in I_0$, and $h \in \bar J_i$, we set
\begin{equation*}
\nu_i^d (h) = \sup \{ u(h) \ | \ u \in \mathcal S_i^- \cap C(\bar J_i), \ u(0) = d \},
\end{equation*}
and observe that $\nu_i^d \in \mathcal S_i \cap C(\bar J_i)$ and $\nu_i^d (0) = d$.

\begin{thm} \label{exist}
Let $(d_0, d_1, d_2, d_3) \in \mathbb R^4$. 
The problem \eqref{HJ} has
a viscosity solution 
$(u_1, u_2, u_3)$ $\in C(\bar J_1) \times C(\bar J_2) \times C(\bar J_3)$
if and only if
\begin{align} \label{condi: admissible}
\begin{cases}
(d_0, d_1, d_2, d_3) \in I_0 \times I_1 \times I_2 \times I_3, \\
\min_{i \in \{ 1, 2, 3 \}} \rho_i^{d_i} (0) \geq d_0, \\
\nu_i^{d_0} (h_i) \geq d_i \ \ \ \text{ for all } i \in \{ 1, 2, 3 \}.
\end{cases}
\end{align} 
\end{thm}

\begin{proof}
First, assume that \eqref{condi: admissible} is satisfied. Set
\begin{equation*}
u_i (h) = \rho_i^{d_i} (h) \wedge \nu_i^{d_0} (h) \ \ \ \text{ for } h \in \bar J_i \text{ and } i \in \{ 1, 2, 3 \},
\end{equation*}
observe that $u_i \in \mathcal S_i \cap C(\bar J_i)$, $u_i (h_i) = d_i$, and $u_i (0) = d_0$ for all $i \in \{ 1, 2, 3 \}$,
and conclude that $(u_1, u_2, u_3)$ is a viscosity solution of \eqref{HJ}.

Now, assume that \eqref{HJ} has
a viscosity solution $(u_1, u_2, u_3) \in C(\bar J_1) \times C(\bar J_2) \times C(\bar J_3)$. 
Obviously, $(d_0, d_1, d_2, d_3) \in I_0 \times I_1 \times I_2 \times I_3$,
$\rho_i^{d_i} \geq u_i$ and $\nu_i^{d_0} \geq u_i$ on $\bar J_i$ for all $i \in \{ 1, 2, 3 \}$.
Moreover we see that $\rho_i^{d_i} (0) \geq u_i (0) = d_0$
and $\nu_i^{d_0} (h_i) \geq u_i = d_i$ for all $i \in \{ 1, 2, 3 \}$.
Thus \eqref{condi: admissible} is valid.  
\end{proof}

We set
\begin{equation*}\mathcal{D} = \{ (d_0, d_1, d_2, d_3) \in \mathbb R^4 \ | \
\text{ \eqref{condi: admissible} is satisfied } \},
\end{equation*}
and 
\begin{align*}
\mathcal D_0 = \{ (d_0, &d_1, d_2, d_3) \in \mathbb R^4 \mid \\&\text{ there exists $a > 0$ such that }
(d_0 + a, d_1 + a, d_2 + a, d_3 + a)\in \mathcal D \}. 
\end{align*}

\begin{lem} \label{D_0}
Let $(d_0, d_1, d_2, d_3) \in \mathcal D_0$.
Then 
\begin{equation*}
d_0 < \min_{i \in \{ 1, 2, 3 \}} \rho_i^{d_i} (0). 
\end{equation*}
\end{lem}

\begin{proof}
Choose $ a > 0$ so that $(d_0 + a, d_1 + a, d_2 + a, d_3 + a) \in \mathcal D$. 
Fix any $i \in \{ 1, 2, 3 \}$ and note that
the function $u := \rho_i^{d_i + a} - a$ is a subsolution of 
\begin{equation*}
\lambda u + \overline G_i (h, u') = -\lambda a \ \ \ \text{ in } J_i.
\end{equation*}
Select a smooth function $\psi \in C^1(\bar J_i)$ so that 
$\psi (h) = 1$ in a neighborhood of $0$ and $\psi (h) =0$ in a neighborhood of $h_i$.
Accordingly, $\psi'$ is supported in $J_i$. 
Let $\varepsilon > 0$ and consider the function $u_\varepsilon := u + \varepsilon \psi$.
Let $M > 0$ be a Lipschitz bound of $u$ in $\mathrm{supp} \, \psi'$, 
note that $\overline G_i$ is uniformly continuous in $\mathrm{supp} \, \psi' \times [-M, M]$, and 
observe that if $\varepsilon > 0$ is sufficiently small, then 
\begin{equation*}
\overline G_i (h, u'(h) + \varepsilon \psi' (h)) \leq \overline G_i (h, u'(h)) + \frac{\lambda a}{2} \ \ \ \text{ for a.e. } h \in J_i.
\end{equation*}
Thus, if $\varepsilon \in (0, a/2)$ is sufficiently small, then we have
\begin{equation*}
\lambda u_\varepsilon + \overline G_i (h, u'_\varepsilon) \leq 0 \ \ \ \text{ in } J_i
\end{equation*}
in the viscosity sense.
Fix such $\varepsilon > 0$ and observe that
$u_\varepsilon (h_i) = u(h_i) = d_i$ and,
by the definition of $\rho_i^{d_i}$, that $\rho_i^{d_i} \geq u_\varepsilon$ on $\bar J_i$.
Since $u_\varepsilon (0) = \rho_i^{d_i + a} (0) - a + \varepsilon > \rho^{d_i + a} (0) - a$, 
we get $\rho_i^{d_i} (0) > \rho_i^{d_i + a} (0) - a$. 
Since $(d_0 + a, d_1 + a, d_2 + a, d_3 + a) \in \mathcal D$,
we have $d_0 + a \leq \rho_i^{d_i+a} (0)$.
Combining this with the inequality above, we get $d_0 < \rho_i^{d_i} (0)$.
This is true for all $i \in \{ 1, 2, 3 \}$ and the proof is complete.
\end{proof}

\begin{thm}
For any $(d_0, d_1, d_2, d_3) \in \mathcal D_0$, \emph{(G5)} and \emph{(G6)} hold for 
some boundary data $g^\varepsilon$. 
\end{thm}

\begin{proof}
Choose $a>0$ so that $(d_0+3a,d_1+3a,d_2+3a,d_3+3a)\in\mathcal D$.
Set $a_i = d_i + a$ for $i \in \{ 0, 1, 2, 3 \}$ and
\begin{equation*}
v_i (h) = \rho_i^{a_i+ a} (h) \wedge \nu_i^{a_0 + a} (h) - a \ \ \ \text{ for } h \in \bar J_i \text{ and } i \in \{ 1, 2, 3 \},
\end{equation*}
and observe that $v_i (h_i) = a_i$ and $v_i (0) = a_0$ for all $i \in \{ 1, 2, 3 \}$.
Observing that $\rho_i^{a_i+ a} \wedge \nu_i^{a_0 + a} \in \mathcal S_i \cap C(\bar J_i)$ and
noting that $v_i$ are locally Lipschitz continuous in $\bar J_i \setminus \{ 0 \}$,
we may choose $c > 0$ so that 
\begin{equation} \label{ineq: v_i}
\lambda v_i (h) + \overline G_i (h, v_i' (h)) \leq -c \ \ \ \text{ for a.e. } h \in J_i \text{ and all } i \in \{ 1, 2, 3 \}.
\end{equation}
Noting that $\rho_i := \rho_i^{a_i + a} - a \in \mathcal S_i^- \cap C(\bar J_i)$ for all $i \in \{ 1, 2, 3 \}$,
we see that
\begin{equation*}
\lambda \rho_i (h) + \min_{q \in \mathbb R} \overline G_i (h, q) \leq 0 \ \ \ 
\text{ for all } h \in J_i \text{ and } i \in \{ 1, 2, 3 \},
\end{equation*}
and, hence, using Lemma \ref{lim-min-barG_i}, we get
\begin{equation*}
\lambda \min_{i \in \{ 1, 2, 3 \}} \rho_i (0) + G(0, 0) \leq 0.
\end{equation*}
Since $(a_0 + a, a_1 + a, a_2 + a, a_3 + a) \in \mathcal D_0$,
by Lemma \ref{D_0}, we see that
\begin{equation} \label{d_0+a}
a_0 < \min_{i \in \{ 1, 2, 3 \}} \rho_i (0),
\end{equation}
from which, we may choose $\eta > 0$ so that
\begin{equation*}
\lambda a_0 + G(0, 0) \leq -\eta. 
\end{equation*}
By Lemma \ref{key}, there exist
$\delta \in (0, h_0)$ and $\psi \in C^1 \big(\overline{\Omega (\delta)} \big)$ such that
\begin{equation*}
-b \cdot D\psi + G(x, 0) \leq G(0, 0) + \eta \ \ \ \text{ on } \overline{\Omega (\delta )}.
\end{equation*}
According to \eqref{d_0+a}, by applying Lemma \ref{nu_i^d} with $d = a_0$ and $d_0 = \min_{i = 1, 2, 3} \rho_i (0)$ and
by replacing $\delta \in (0, h_0)$ by a smaller number if necessary,
we may assume that
\begin{equation*}
v_i (h) > a_0 \ \ \ \text{ for all } h \in J_i \cap [-\delta, \delta ] \text{ and } i \in \{ 1, 2, 3 \}. 
\end{equation*}
Now we may choose $\gamma \in \big( a_0, \min_{i\in\{1,2,3\}}\rho_i(0) \big)$ so that
\begin{equation*}
\gamma < \min_{i \in \{ 1, 2, 3 \}} v_i ((-1)^i \delta ) \ \ \ \text{ and } \ \ \ \lambda \gamma + G(0, 0) < -\eta.
\end{equation*}

Setting
\begin{equation*}
w_0^\varepsilon (x) = \gamma - a + \varepsilon \psi (x) \ \ \ \text{ for } x \in \overline{\Omega (\delta )} \text{ and } \varepsilon \in (0, 1),
\end{equation*}
and computing as in \eqref{ineq: w^ep}, we see that there exists
$\varepsilon_0 \in (0, 1)$ such that if $\varepsilon \in (0, \varepsilon_0)$, then 
\begin{equation} \label{ineq: w_0^ep}
\lambda w_0^\varepsilon (x) - \cfrac{b(x) \cdot Dw_0^\varepsilon (x)}{\varepsilon} + G(x, Dw_0^\varepsilon (x)) \leq 0 \ \ \ \text{ for all } x \in \Omega (\delta ).
\end{equation}

Next, note that
there exist $\delta_i \in (0, \delta)$, with $i \in \{ 1, 2, 3 \}$, such that
\begin{equation*}
v_i ((-1)^i \delta_i ) = \gamma \ \ \ \text{ and } \ \ \ v_i (h) < \gamma
\end{equation*}
for all $h \in J_i \cap (-\delta_i, \delta_i )$ and $i\in\{1,2,3\}$.
Set
\begin{equation*}
J_i (h) = J_i \cup (h_i - h, h_i + h) \ \ \ \text{ for } h > 0 \text{ and } i \in \{ 1, 2, 3 \}.
\end{equation*}
Combining \eqref{ineq: v_i} and the fact that there exist $p_i \in \mathbb R$, 
with $i \in \{ 1, 2, 3 \}$, such that $\lim_{J_i \ni h \to h_i} v_i' (h) 
= p_i$ along a subsequence for all $i \in \{ 1, 2, 3 \}$, we have
\begin{equation*}
\lambda v_i (h_i) + \overline G_i (h_i, p_i) \leq -c \ \ \ \text{ for all } i \in \{ 1, 2, 3 \},
\end{equation*}
and, hence, for each $i \in \{ 1, 2, 3 \}$ and for some $r_i \in (0, \delta_i/2)$,
we can extend the domain of definition of $v_i$ to $\overline{J_i (r_i)}$ so that
\begin{equation*}
v_i (h) + \overline G_i (h, v_i' (h)) \leq 0 \ \ \ \text{ for a.e. } h \in J_i (r_i),
\end{equation*}
by setting
\begin{equation*}
v_i (h) = v_i (h_i) + p_i (h - h_i) \ \ \ \text{ for } h \in [h_i - r_i, h_i + r_i] \setminus \bar J_i.
\end{equation*}
Here we have used the fact that, in view of the definition of $\overline G_i$,
we may assume that, for each $i \in \{ 1, 2, 3 \}$, $\overline G_i$ is defined in $\overline{J_i (r_i)} \times \mathbb R$.

For $i \in \{ 1, 2, 3 \}$, $r \in (0, r_i/2)$, and $h \in J_i (r_i/2) \setminus [-\delta_i/2, \delta_i/2]$, set
\begin{equation*}
v_i^r (h) = \varphi_r \ast v_i (h), 
\end{equation*} 
where $\varphi_r (h) = (1/r)\varphi ((1/r)h)$ and
$\varphi \in C^\infty (\mathbb R)$ is a standard mollification kernel, that is,
$\varphi \geq 0$, supp $\varphi \subset [-1, 1]$, and $\int_\mathbb R \varphi \, dx = 1$.

Due to the local Lipschitz continuity of 
$v_i$ in $\overline{J_i (r_i)} \setminus \{ 0 \}$,
there exist $C_i > 0$, with $i \in \{ 1, 2, 3 \}$, such that
\begin{equation*}
|v_i' (h)| \leq C_i
\end{equation*} 
for a.e. $h \in \overline{J_i (r_i)} \setminus (-(\delta_i - r_i)/2, 
(\delta_i - r_i)/2)$. 
For $i \in \{ 1, 2, 3 \}$, let $m_i$ be a modulus of $\overline G_i$
on $\overline{J_i (r_i)} \setminus (-(\delta_i - r_i)/2, (\delta_i -r_i)/2) \times [-C_i, C_i]$.

Fix any $i \in \{ 1, 2, 3 \}$, $r \in (0, r_i/2)$, and $h \in J_i \setminus [-\delta_i/2, \delta_i/2]$, and compute that
\begin{align*}
0 &\geq \lambda \varphi_r \ast v_i (h) + \varphi_r \ast \overline G_i (\cdot, v_i' (\cdot )) (h) \\
   &= \lambda v_i^r (h) + \int_{h - r}^{h + r} \varphi_r (h - s) \overline G_i (s, v_i' (s)) \, ds \\
   &\geq \lambda v_i^r (h) + \int_{h - r}^{h + r} \varphi_r (h - s) \big( \overline G_i (h, v_i' (s)) - m_i (r) \big) \, ds \\
   &\geq \lambda v_i^r (h) + \overline G_i (h, \varphi_r \ast v_i' (h)) - m_i (r) \\
   &= \lambda v_i^r (h) + \overline G_i (h, (v_i^r)' (h)) - m_i (r).
\end{align*}
Here we have used Jensen's inequality in the third inequality. 
Moreover we have
\begin{equation*}
|v_i (h) - v_i^r (h)| \leq \int_{h - r}^{h + r} |v_i (h) - v_i (s)| \varphi_r (h - s) \, ds \leq C_i r,
\end{equation*}
and, hence, we get
\begin{equation*} \label{eq: v_i-v_i^r}
\lambda v_i (h) + \overline G_i (h, (v_i^r)' (h)) \leq m_i (r) + \lambda C_i r
\end{equation*}
for all $h \in J_i \setminus \left[-\delta_i/2,\, \delta_i/2 \right]$, 
$r\in(0,\,r_i/2)$, and $i \in \{ 1, 2, 3 \}$.

Fix a small $r \in (0, r_i/2)$.
For $i \in \{ 1, 2, 3 \}$, we define the function $g_i \in C(\overline \Omega_i \setminus \Omega_i (\delta_i/2))$ by
\begin{equation*}
g_i (x) = G(x, (v_i^r)' \circ H(x)DH(x)),
\end{equation*}
and choose $f_i \in C^1 (\overline \Omega_i \setminus \Omega_i (\delta_i/2))$ so that
\begin{equation*}
|g_i (x) - f_i (x)| < r \ \ \ \text{ for all } x \in \overline \Omega_i \setminus \Omega_i \left( \frac{\delta_i}{2} \right).
\end{equation*}

Let $\tau_i$ are the functions defined in \eqref{def-tau_i},
and, for $i \in \{ 1, 2, 3 \}$,
let $\psi_i$ be the function on $\overline \Omega_i \setminus \Omega_i (\delta_i/2)$ defined by
\begin{equation*}
\psi_i (x) = \int_0^{\tau_i (x)} \Big( f_i (X(t, x)) - \bar f_i (x) \Big) \, dt,
\end{equation*}
where
\begin{equation*}
\bar f_i (x) := \frac{1}{T_i \circ H(x)} \int_0^{T_i \circ H(x)} f_i (X(t, x)) \, dt. 
\end{equation*}

Recalling that 
$X \in C^1 (\mathbb R \times \mathbb R^2)$, that  
$\tau_i \in C^1 \left( \left( \overline \Omega_i \setminus c_i (0) \right) \setminus l_i \right)$ and $T_i \in C^1 (\bar J_i \setminus \{ 0 \})$ for all $i\in\{1,2,3\}$,
and that the definition of $\tilde \tau_i$,
it is clear that $\psi_i \in C^1 (\overline \Omega_i \setminus \Omega_i (\delta_i/2))$ for all $i \in \{ 1, 2, 3 \}$.
By using the dynamic programming principle,
we see that
\begin{equation*}
-b(x) \cdot D\psi_i(x) = -f_i(x) + \bar f_i(x) 
\ \ \ \text{ for all } 
x \in \overline \Omega_i \setminus \Omega_i \left( \frac{\delta_i}{2} \right)
\text{ and } i \in \{ 1, 2, 3 \}.
\end{equation*}

For $i \in \{ 1, 2, 3 \}$, $\varepsilon \in (0, \varepsilon_0)$,
and $x \in \overline \Omega_i \setminus \Omega_i (\delta_i/2)$, we set
\begin{equation*}
w_i^\varepsilon (x) = v_i \circ H(x) - a + \varepsilon \psi_i (x), 
\end{equation*}
and observe that $w_i^\varepsilon \in \mathrm{Lip} \big( \overline \Omega_i \setminus \Omega_i (\delta_i/2) \big)$.
Fix any $i \in \{ 1, 2, 3 \}$ and $\varepsilon \in (0, \varepsilon_0)$,
and compute that, for almost every $x \in \Omega_i \setminus \overline{\Omega_i (\delta_i/2)}$,
\begin{align*}
\lambda &w_i^\varepsilon (x) - \frac{b (x) \cdot Dw_i^\varepsilon (x)}{\varepsilon} + G(x, Dw_i^\varepsilon (x)) \\
           &= \lambda v_i \circ H(x) - \lambda a + \varepsilon \lambda \psi_i (x) - b(x) \cdot D\psi_i (x) \\
           & \quad + G\big( x, v_i' \circ H(x)DH(x) + \varepsilon D\psi_i (x) \big) \\
           &= \lambda v_i \circ H(x) - \lambda a + \varepsilon \lambda \psi_i (x) - f_i (x) + \bar f_i (x) \\ 
           & \quad + G\big( x, v_i' \circ H(x)DH(x) + \varepsilon D\psi_i (x) \big) \\
           &< \lambda v_i \circ H(x) - \lambda a + \varepsilon \lambda \psi_i (x) - g_i (x) + r \\
           & \quad + \frac{1}{T_i \circ H(x)} \int_0^{T_i \circ H(x)} \Big( g_i (X(t, x)) + r \Big) \, dt + G(x, v_i' \circ H(x)DH(x) + \varepsilon D\psi_i (x)) \\
           &\leq  - \lambda a + \varepsilon \lambda \psi_i (x) + m_i (r) + (2 + \lambda C_i)r \\ 
           & \quad - G\big( x, (v_i^r)' \circ H(x)DH(x)\big) + G\big( x, v_i' \circ H(x)DH(x) + \varepsilon D\psi_i (x)\big),
\end{align*}
from which,
by replacing $r \in (0, r_i)$ and $\varepsilon_0 \in (0, 1)$ by smaller numbers if necessary,
we may assume that if $\varepsilon \in (0, \varepsilon_0)$, then,
\begin{equation} \label{ineq: w_i^ep}
\lambda w_i^\varepsilon (x) - \frac{b (x) \cdot Dw_i^\varepsilon (x)}{\varepsilon} + G(x, Dw_i^\varepsilon (x)) \leq 0,
\end{equation}
for a.e. $x \in \Omega_i \setminus \overline{\Omega_i (\delta_i/2)}$ and for all $i \in \{ 1, 2, 3 \}$. 

Observing that
\begin{align}\label{W_0-bound1} 
\lim_{\varepsilon \to 0} w_0^\varepsilon (x) &= \gamma - a \ \ \ \text{ uniformly for } x \in \overline{\Omega (\delta )}, \\ \label{W_0-bound2}
 \lim_{\varepsilon \to 0} w_i^\varepsilon (x) &= v_i \circ H(x) - a \ \ \ \text{ uniformly for } x \in \overline{\Omega}_i \setminus \Omega_i (\delta_i/2 ),
\end{align}
and
\begin{equation}\label{W_0-bound3} 
v_i ((-1)^i \delta ) > \gamma \ \ \ \text{ and } \ \ \ v_i \left( \frac{(-1)^i \delta_i}{2} \right) < \gamma \ \ \ \text{ for all } i \in \{ 1, 2, 3 \},
\end{equation}
by replacing $\varepsilon_0 \in (0, 1)$ by a smaller number if necessary,
we may assume that if $\varepsilon \in (0, \varepsilon_0)$,
then, for all $i \in \{ 1, 2, 3 \}$,
\begin{equation*}
w_i^\varepsilon > w_0^\varepsilon \ \ \ \text{ on } \ c_i ((-1)^i \delta ) \ \ \ \text{ and } \ \ \ w_i^\varepsilon < w_0^\varepsilon \ \ \ \text{ on } c_i \left(\frac{(-1)^i \delta_i}{2} \right).
\end{equation*}

For $\varepsilon \in (0, \varepsilon_0)$, we set
\begin{align*}
w^\varepsilon (x) = 
\begin{cases}
w_0^\varepsilon (x)                                 \ \ \ &\text{ if } x \in 
\overline{\Omega_i (\delta_i/2 )} \text{ and } i \in \{ 1, 2, 3 \}, \\
w_0^\varepsilon (x) \vee w_i^\varepsilon (x) \ \ \ &\text{ if } x \in \Omega_i (\delta ) \setminus \overline{\Omega_i (\delta_i/2 )} \text{ and } i \in \{ 1, 2, 3 \}, \\
w_i^\varepsilon (x)                                  \ \ \ &\text{ if } x \in \overline \Omega_i \setminus \Omega_i (\delta ) \text{ and } i \in \{ 1, 2, 3 \}.
\end{cases}
\end{align*}
Noting that $w^\varepsilon \in \mathrm{Lip} (\overline \Omega)$,
by \eqref{ineq: w_0^ep} and \eqref{ineq: w_i^ep},
we have, for all $\varepsilon \in (0, \varepsilon_0)$,
\begin{equation} \label{subsol-w^ep} 
w^\varepsilon - \frac{b \cdot Dw^\varepsilon}{\varepsilon} + G(x, Dw^\varepsilon) \leq 0 \ \ \ \text{ in } \Omega,
\end{equation}
in the viscosity sense.

Now we set
\begin{equation*}
g^\varepsilon (x) = w^\varepsilon (x) \ \ \ \text{ for } \varepsilon \in (0, \varepsilon_0 ) \text{ and } x \in \partial \Omega.
\end{equation*}

Now we intend to show that (G5) and (G6) hold. 
For $i \in \{ 1, 2, 3 \}$, $\varepsilon \in (0, \varepsilon_0)$, and $x \in \overline \Omega_i \setminus \Omega_i (\delta_i/2)$
and for some constant $C > 0$, set
\begin{equation*}
W_i^\varepsilon (x) = w_i^\varepsilon (x) + (-1)^i C (h_i - H(x)),
\end{equation*}
and observe that $W_i^\varepsilon \in \mathrm{Lip} \big( \overline \Omega_i \setminus \Omega_i (\delta_i/2) \big)$,
$W_i^\varepsilon = g^\varepsilon$ on $\partial_i \Omega$, and
$W_i^\varepsilon \geq w_i^\varepsilon$ on $\overline \Omega_i \setminus \Omega_i (\delta_i/2)$.
Fix any $i \in \{ 1, 2, 3 \}$ and $\varepsilon \in (0, \varepsilon_0)$
and compute that, for almost every $x \in \Omega_i \setminus \overline{\Omega_i (\delta_i/2)}$,
\begin{align*}
\lambda &W_i^\varepsilon (x) - \frac{b (x) \cdot DW_i^\varepsilon (x)}{\varepsilon} + G(x, DW_i^\varepsilon (x)) \\
           &= \lambda v_i \circ H(x) - \lambda a + \varepsilon \lambda \psi_i (x) + (-1)^i \lambda C (h_i - H(x)) - b(x) \cdot D\psi_i (x) \\
           & \quad + G \big( x, v_i' \circ H(x)DH(x) + \varepsilon D\psi_i (x) - (-1)^i CDH(x) \big),
\end{align*}
from which, in view of the coercivity of $G$,  
by replacing $C > 0$, independently of $\varepsilon \in (0, \, \varepsilon_0)$, by a larger number if necessary, we conclude that
\begin{equation*}
\lambda W_i^\varepsilon (x) - \frac{b (x) \cdot DW_i^\varepsilon (x)}{\varepsilon} + G(x, DW_i^\varepsilon (x)) \geq 0.
\end{equation*}
We note that
\begin{equation*}
\lim_{\varepsilon \to 0+} W_i^\varepsilon (x) = W_i \circ H (x) := v_i \circ H (x) - a + (-1)^i C(h_i - H(x))
\end{equation*}
uniformly for $x \in \overline \Omega_i \setminus \Omega_i (\delta_i/2)$ and that $W_i \in \mathrm{Lip} \big( \bar J_i \setminus (-\delta_i/2, \, \delta_i/2) \big)$ and $W_i (h_i) = d_i$.

Next, set $M = \max_{\overline \Omega} |G(x, 0)|$ and
\begin{equation*}
W_0^\varepsilon (x) = (M/\lambda) \vee \max_{\overline \Omega} w^\varepsilon \ \ \ \text{ for } x \in \overline \Omega \text{ and } \varepsilon \in (0, \varepsilon_0).
\end{equation*}
It is now easily seen that for all $\varepsilon \in (0, \varepsilon_0)$ and $x \in \overline \Omega$,
$W_0^\varepsilon (x) \geq w^\varepsilon (x)$ and 
\begin{equation*}
\lambda W_0^\varepsilon (x) - \frac{b (x) \cdot DW_0^\varepsilon (x)}{\varepsilon} + G(x, DW_0^\varepsilon (x)) \geq 0.
\end{equation*}
Moreover, combining \eqref{W_0-bound1}, \eqref{W_0-bound2}, and \eqref{W_0-bound3},
we see from the definition of $w^\varepsilon$ that
\begin{equation*}
\lim_{\varepsilon \to 0+} W_0^\varepsilon (x) = W_0 (x) := (M/\lambda) \vee \max_{i \in \{ 1, 2, 3 \}} \max_{\overline \Omega_i \setminus \Omega_i (\delta_i/2)} (v_i \circ H - a)
\end{equation*}
uniformly for $x \in \overline \Omega$.

Now, we note that
\begin{equation*}
W_i^\varepsilon (x) = w_i^\varepsilon (x) = g^\varepsilon (x) \leq W_0^\varepsilon (x)
\end{equation*}
for all $x \in \partial_i \Omega$, $\varepsilon \in (0, \varepsilon_0)$, and $i \in \{ 1, 2, 3 \}$.
By replacing $C > 0$ by a larger number if necessary, we may assume that
\begin{equation*}
W_i^\varepsilon (x) > W_0^\varepsilon (x)
\end{equation*}
for all  $x \in \Omega_i (\delta)$, $\varepsilon \in (0, \varepsilon_0)$, and $i \in \{ 1, 2, 3 \}$.

For $\varepsilon \in (0, \varepsilon_0)$, we set
\begin{equation*}
W^\varepsilon (x) = W_0^\varepsilon (x) \wedge W_i^\varepsilon (x) \ \ \  \text{ if } x \in \overline \Omega_i \text{ and } i \in \{ 1, 2, 3 \},
\end{equation*}
and observe that $W^\varepsilon \in \mathrm{Lip} (\overline \Omega)$, $W^\varepsilon$ is a viscosity supersolution of \eqref{epHJ},
$W^\varepsilon = g^\varepsilon$ on $\partial \Omega$, $W^\varepsilon \geq w^\varepsilon$ on $\overline \Omega$, and 
\begin{equation*}
\lim_{\varepsilon \to 0+} W^\varepsilon (x) = W(x) := W_0 (x) \wedge W_i \circ H (x)
\end{equation*}
uniformly for  $x \in \overline \Omega_i$ and for all $i \in \{ 1, 2, 3 \}$.
It is obvious that $W_0 \wedge W_i \circ H \in \mathrm{Lip} (\overline \Omega_i)$
and $W_0 \wedge W_i \circ H = d_i$ on $\partial_i \Omega$ for all $i \in \{ 1, 2, 3 \}$.

Now, by Perron's method and the comparison principle,
there exists a unique viscosity solution $u^\varepsilon \in C(\overline \Omega)$ of \eqref{epHJ} satisfying $u^\varepsilon = g^\varepsilon$ on $\partial \Omega$
such that  
\begin{equation*}
w^\varepsilon \leq u^\varepsilon \leq W^\varepsilon \ \ \ \text{ on } \overline \Omega,
\end{equation*}
and, hence, in view of Proposition \ref{viscosity-sol}, (G5) holds.
Also, the inequality above yields, for all $i \in \{ 1, 2, 3 \}$,
\begin{equation*}
v_i \circ H - a \leq v^- \leq v^+ \leq W_i \circ H
\end{equation*}
in a neighborhood of $\partial \Omega_i$,
and, therefore, for all $i \in \{ 1, 2, 3 \}$ and $x \in \partial_i \Omega$,
\begin{align*}
d_i = \lim_{\Omega_i \ni y \to x} v_i \circ H(y) - a &\leq \lim_{\Omega_i \ni y \to x} v^- (y) \\
                                                                 &\leq \lim_{\Omega_i \ni y \to x} v^+ (y) \leq \lim_{\Omega_i \ni y \to x} W_i \circ H (y) = d_i.
\end{align*}
This implies (G6). The proof is complete.

\end{proof}

\subsection*{Acknowledgments}

The author would like to thank Prof. Hitoshi Ishii
for many helpful comments and discussions about
the singular perturbation problem of Hamilton-Jacobi equations treated here.
The author would like to thank the anonymous referee for careful reading, 
valuable comments and pointing out several errors.


(T. Kumagai) 
Department of Pure and Applied Mathematics, 
Graduate School of Fundamental Science and Engineering, Waseda University, Shinjuku, Tokyo 69-8050 Japan

E-mail: kumatai13@gmail.com

\end{document}